\documentclass[11pt,english]{article}
\usepackage{xr}
\usepackage{makecell}
\usepackage{tabularx}
\usepackage{subcaption}
\usepackage{caption}
\usepackage{graphicx}
\usepackage{threeparttable}
\usepackage{ulem} 
\usepackage{xcolor}
\usepackage{setspace}
\usepackage{natbib}
\usepackage{booktabs}
\usepackage{ulem}
\usepackage{makecell,rotating,multirow,diagbox} 
\usepackage{amssymb}
\usepackage{mathrsfs}
\usepackage{amsmath,bm}

\usepackage{amsbsy}
\usepackage{setspace}
\usepackage{soul}
\usepackage{array}
\usepackage{booktabs}
\usepackage{enumerate}
\usepackage{enumitem}

\usepackage{pifont}
\newcommand{\cmark}{\ding{51}}
\newcommand{\xmark}{\ding{55}}
\usepackage{longtable}
\usepackage{setspace}
\usepackage{appendix}
\usepackage{geometry}
\usepackage{authblk}
\usepackage{url}
\emergencystretch=\maxdimen
\usepackage[english]{babel}
\usepackage[font=normalsize, labelfont=bf]{caption}
\usepackage[font=normalsize, labelfont=bf]{subcaption}
\usepackage[ruled,vlined]{algorithm2e}
\usepackage{amsthm}

\newtheorem{proposition}{Proposition}

\usepackage{natbib}
 \bibpunct[, ]{(}{)}{,}{a}{}{,}

\begin{document}

\title{Dynamic Carbon Intensity Indicator (CII) Management in Stochastic Tramp Shipping Market}
\author[a]{Hanyu Cheng}
\author[a]{Liangqi Cheng}
\author[a]{Xiwen Bai\thanks{Corresponding author: xiwenbai@mail.tsinghua.edu.cn}}
\affil[a]{Department of Industrial Engineering, Tsinghua University, 30 Shuangqing Road, Beijing 100084, China}
\date{}
\maketitle

\begin{abstract}
In the maritime sector, tramp shipping companies manage fleets to maximize profit while navigating market uncertainties. The International Maritime Organization (IMO) recently introduced the Carbon Intensity Indicator (CII) to reduce greenhouse gas emissions, further complicating deployment decisions. This paper introduces a novel two-stage stochastic programming model for long-term fleet deployment under market uncertainty and CII regulation. It is the first to integrate key operational uncertainties\textemdash fuel prices, freight rates, and cargo demand\textemdash into a unified tactical planning framework under CII regulation, simultaneously optimizing routing, cargo allocation, and speed. Furthermore, we develop an novel efficient heuristic algorithm that reliably converges to solutions within a 5\% optimality gap, enabling practical decision-support under uncertainty. Numerical analysis highlights two key findings based on our model: (1) It uncovers the ``CII paradox,'' a critical counterintuitive phenomenon where the present Supply-based CII regulation may increase total emissions significantly and drastically reduce profits. This challenges the conventional wisdom that stricter carbon-intensity rules invariably reduce emissions. (2) It demonstrates the advantage of stochastic modeling, showing that accounting for future uncertainties significantly narrows the revenue gap with perfect-foresight solutions, thereby offering superior economic performance over deterministic approaches. Collectively, these results deepen the understanding of environmental regulation's operational impacts and pave the way for more effective and sustainable fleet management strategies.

\end{abstract}

\newpage

\section{Introduction}
Maritime transportation serves as the cornerstone of international trade, carrying over 80\% of global goods by volume \citep{unctad2021review}. The shipping industry can be broadly categorized into liner shipping and tramp shipping services.
In the tramp shipping market, companies dynamically manage their fleets to transport cargo, aiming to minimize costs while fulfilling contractual obligations. They also seek additional revenue by leveraging the spot market \citep{wang2021carrier}, where strategic optimization of fleet deployment decisions can significantly enhance profits. These fleet deployment decisions typically involve determining how to assign ships to different routes which are composed of several ports in a certain sequence and at what times, deciding which cargo each ship will carry and the volume of the cargo (comprising both contractual and spot market cargo), and the speeds at which they will sail in each trip from origin to destination. These decisions are hereby referred to as route, cargo, and speed selections. However, such fleet deployment decisions are complicated by uncertainties in market conditions, including fluctuating contractual demands and spot market volatility. For instance, some contracts depend on customer production levels rather than fixed quantity, while spot market demand varies with commodity-specific conditions. These uncertainties necessitate meticulous planning of fleet routing and deployment strategies, a topic that has drawn substantial research attention, including contributions by \cite{gelareh2010novel} and \cite{andersson2015integrated}, among others.

In addition to market-driven challenges, the shipping industry has faced an increasingly stringent and multi-layered regulatory framework aimed at curbing maritime emissions in recent years. Internationally, the International Maritime Organization (IMO) has implemented a series of mandatory measures, including the global 0.5\% sulphur cap \citep{imo2020sulphurcap}, as well as the Energy Efficiency Existing Ship Index (EEXI) and the CII for the majority of ships \citep{mepc2021336}, with 24,653 ships reporting their CII ratings in 2023 \citep{imo2023mepc82}. Regionally, the European Union has incorporated maritime transport into its Emissions Trading System (EU ETS) to price carbon \citep{eu_directive_2023_959}. Among these regulatory schemes, the CII has been a hot topic. Recent research has evaluated its regulatory effectiveness \citep{wang2021paradox}, its assessments on different ships \citep{bayraktar2023scenario}, and its implications for operational decisions in shipping \citep{cheng2025cargo}. It differs fundamentally from other mechanisms such as the EU ETS and the EEXI. While the EU ETS emphasizes a market-based cap-and-trade approach to regulate total emissions \citep{ellerman2010pricing}, and the EEXI defines a technical baseline targeting the design efficiency of existing ships \citep{bayraktar2023scenario}, the CII instead directly links emissions to actual transport performance, serving as a standardized indicator of carbon intensity through ships' real operational data. Therefore, unlike studies on the EU ETS that primarily focus on analyzing market trading dynamics \citep{hintermann2016price,abrell2011assessing,zhang2010overview}, and those on the EEXI that examine its potential influence on technological innovation and ship retrofitting \citep{lee2024analysis,bayramouglu2024effects,psillas2022effective}, our work focus on integrating the CII into decision-making from the perspective of shipping companies.

Although the CII encourages operators to enhance energy efficiency through operational measures and adopting alternative fuels or advanced technologies, the substantial investment required for these technologies often renders them inaccessible. As a result, operational optimization becomes a more immediate and practical lever for most shipping companies to comply with CII regulations in the near term \citep{yuan2023operational}. However, this adds new constraints to a shipping company's decision-making process and may lead to decisions that differ from those focused solely on economic profitability, as it balances environmental goals with the need for operational efficiency. For instance, a ship might choose to reduce its speed to decrease carbon emissions, in order to meet the annual CII requirement. Consequently, this would result in a longer transit time for delivering a cargo compared to before, thereby leading to a decrease in the amount of cargo that can be transported within a specific time frame \citep{hua2024speed}. Furthermore, the CII's annual evaluation makes it challenging to optimize operations amid uncertain market conditions. Shipping companies must not only make decisions based on current and predicted market conditions but also consider long-term environmental perspectives to meet the annual CII targets. This requires them to adjust their strategies to ensure compliance with CII, regardless of future market fluctuations, introducing a higher level of complexity into the decision-making process. With the incorporation of CII, short-term decisions now directly influence long-term decision-making process, as the two are interconnected through the CII constraints. When short-term decisions are made, in collaboration with uncertain future market conditions, they directly limit the feasible decision space for long-term decisions in the future to fulfill CII requirements while maintaining earnings as much as possible. As a result, optimizing the deployment decisions becomes significantly more challenging.

Consequently, tramp shipping operators must navigate a central problem: optimizing fleet deployment amid market volatility while simultaneously complying with the stringent CII regulation. While stochastic programming is well-established for handling market uncertainties in fleet deployment \citep{pantuso2016uncertainty, arslan2017bulk}, and recent studies have begun integrating environmental policies into operational models \citep{de2023sailing, wang2021paradox}, a critical gap remains: Existing stochastic models overlook CII's binding constraints \citep{wu2021robust, pantuso2016uncertainty}, while CII-focused models are largely deterministic and cannot address long-term market risks \citep{cheng2025cargo}. This disconnect is critical because the annual CII mechanism intrinsically couples decisions under uncertainty with the long-term environmental target, a dynamic previously unmodeled.

To bridge this gap, we propose a two-stage stochastic programming model that simultaneously optimizes fleet deployment under market uncertainty and CII compliance. The model focuses on route selection, cargo selection, and sailing speed optimization, accounting for variations in fuel prices, market demand, and freight rates. Market scenarios are generated based on economic trends to reflect actual fluctuations. This model addresses the challenges mentioned above in two ways: First, it incorporates the annual CII as a constraint, rather than focusing solely on maximizing economic profit, which forces decisions that minimize economic losses while being forced to comply with the CII requirements. Second, the model considers all possible future market scenarios as constraints and parts of objective function and explores the feasible decision space for both the first (short-term) and second (long-term) stages. This ensures that decisions made in the first stage are sufficiently conservative, allowing second-stage have enough space to make decisions to meet the CII requirements while also maximizing economic profit.

We conduct numerical experiments with two forms of CII (Demand-based CII, based on actual cargo carried, and Supply-based CII, which uses ship capacity as a proxy), examining decisions made by ships and comparing the overall profits and carbon emissions of the shipping company under both CII forms, which highlights potential shortcomings in the current CII. Additionally, numerical analysis is performed under different levels of information availability, including cases where the company knows the actual future information, the distribution of future market conditions, and only the expected value of future market conditions, demonstrating the critical role of accessible information in decision-making.

The contribution of our paper can be summarized as follows:

\begin{itemize}
    \label{sec:contribution}
    \item \textbf{Dynamic CII management under long-term uncertainty}: We propose a two-stage stochastic model integrating annual CII constraints into fleet deployment at a tactical level, capturing long-term uncertainties, unlike deterministic, voyage-level approaches in previous literature.
    \item \textbf{Demonstrating the CII paradox in tramp shipping through fleet-level optimization}: Our model reveals how Supply-based CII can unintentionally increase ballast sailing and emissions due to economic-environmental conflicts by modeling realistic fleet operations.
    \item \textbf{Quantifying the value of future information}: We evaluate how market foresight levels impact profitability, showing stochastic models significantly outperform deterministic ones.
\end{itemize}

The structure of the subsequent sections is organized as follows. We review the relevant studies in Section \ref{sec:literature}, and describe the problem in detail in Section \ref{sec:problem_description}. Section \ref{sec:model_formulation} presents the mathematical formulation of the proposed problem. In Section \ref{sec:algorithm}, we develop a heuristic solution method to reduce the computation complexity. The numerical analysis is conducted in Section \ref{sec:case_study}. Finally, conclusions including the limitations of this work and possible avenues for future research are discussed in Section \ref{sec:conclusion}.

\section{Literature review}
\label{sec:literature}

This section reviews related studies on maritime fleet deployment under uncertainties and the incorporation of environmental policies, including the Carbon Intensity Index (CII), in maritime transportation decision-making.

Previous research has extensively examined tramp shipping operations under deterministic market conditions, with a primary focus on short-term fleet deployment and routing decisions. Some studies address routing and bunkering optimization to improve profitability \citep{meng2015tailored,he2024route}, while others incorporate speed optimization into the planning framework \citep{he2024route,tamburini2025rich}. Routing and scheduling problems that include optional cargo selection and cargo splitting have also been studied under deterministic settings \citep{korsvik2011large,cheng2025cargo}. More recently, researchers have explored operational planning under special conditions, such as tidal berth time windows \citep{gao2023tramp} and industrial shipping scenarios with a single export terminal and multiple import terminals \citep{song2025bulk}. These studies provide valuable insights into tramp ship operations under deterministic and operational contexts. However, their focus on relatively short-term and deterministic conditions limits their applicability to long-term decision-making under uncertainty.

When it comes to relatively long-term decision making process in the tramp shipping market, decisions are usually made under economic uncertainty influenced by various factors. Inflation, interest rates, and exchange rates all influence the supply and demand for shipping services \citep{abakah2024asymmetric}, and consequently affects freight rate fluctuations. Besides, variation in fuel prices, a major cost driver, also significantly affect transportation expenses \citep{jia2024resilience,vcech2022marine,xu2011dynamics}. On this basis, a significant body of research has integrated uncertainty into shipping companies' decision-making processes, addressing the long-term ship deployment problem. The two primary approaches for modeling this issue are the robust optimization model and the stochastic optimization model \citep{christiansen2025fifty}. Robust optimization models enhance the resilience of fleet deployment decisions by ensuring feasibility under worst-case realizations of uncertain parameters. In this stream, they typically adopt mixed-integer linear programming (MILP) robust optimization models to handle fleet sizing \citep{wu2021robust} and deployment problems \citep{alvarez2011robust}, where uncertainties in shipping prices, charter rates, and demolition values are captured through bounded uncertainty sets. In contrast, stochastic optimization models explicitly account for probabilistic scenarios of future market conditions. These models optimize fleet renewal (such as sales and chartering) and deployment decisions (such as service frequency and lay-up), under uncertainties in demand, fuel costs, charter costs, and ship values through multi-stage stochastic programming frameworks \citep{pantuso2016uncertainty, arslan2017bulk}. Although some studies were originally developed for liner shipping, their formulations can be readily generalized to tramp operations \citep{pantuso2016uncertainty}, where voyage patterns and market exposure exhibit similar stochastic characteristics. The studies mentioned in this paragraph consider the optimization of decisions regarding fleet deployment and speed in the tramp shipping market, both under deterministic case and stochastic case. However, they do not incorporate environmental policies into their model, thereby failing to capture the regulatory constraints that increasingly shape deployment decisions.

In current research, considerations of environmental policies can be categorized into two types: one emphasizes operational optimization aimed at improving environmental performance, and the other focuses on the evaluation of the effectiveness of environmental policies at a high level. The first category of studies focuses on operational optimization, proposing specific technologies or strategies to address ever-evolving environmental policies. Such studies have incorporated the constraints of the  environmental policies into fleet decision-making processes. From the perspective of shipping operators, previous studies design optimization frameworks for shipping operations, including bunkering to comply with the IMO 2020 sulfur fuel regulations \citep{de2023sailing}, coordinating vessel arrivals with port service capacity to reduce fuel consumption and emissions \citep{sung2022multi}, adjusting engine power settings to achieve fuel and emission savings \citep{wang2021voyage}, and optimizing operational pickup-and-delivery decisions with embedded CII constraints \citep{cheng2025cargo}. The second category of studies focuses on evaluating and optimizing the effectiveness of environmental policies, examining how regulatory measures influence shipping operations and investment decisions. Previous studies investigate policy design and implementation to reduce carbon emissions, including optimizing sulfur emission control areas \citep{zhuge2024shipping}, assessing the short-term effects of bunker levies on tanker market emissions \citep{lagouvardou2022impacts}, analyzing the influence of shipping alliances on green investments and the effectiveness of emission taxes \citep{shang2024would}, and highlighting the potential paradox of the CII where compliance may paradoxically increase emissions \citep{wang2021paradox}. These two research directions contribute to a deeper understanding of how the shipping industry can adapt to environmental policies, balancing efficiency with sustainability goals, and revealing potential improvements for environmental policy making. However, few studies have captured the interplay between long-term market uncertainty and dynamic environmental regulations in deployment models.

\begin{table}[!htbp]
\small
\centering
\renewcommand{\arraystretch}{1.5}
\caption{Research comparison in literature.}
\label{tab:literature_compare}
\resizebox{\textwidth}{!}{%
\begin{tabularx}{\textwidth}{p{2.5cm}cccccc}
\toprule
\textbf{Paper} & \multicolumn{3}{c}{\textbf{Decision}} & \multirow{2}{*}{\textbf{Planning Level}} & \multirow{2}{*}{\textbf{Uncertainty}} & \multirow{2}{*}{\textbf{Environmental Policy}} \\
\cmidrule(lr){2-4}
& Route & Speed & Cargo & & \\
\midrule
\cite{meng2015tailored} & \cmark & \xmark & \cmark & Operational & \xmark & \xmark \\
\cite{he2024route} & \cmark & \cmark & \xmark & Operational & \xmark & \xmark \\
\cite{tamburini2025rich} & \cmark & \cmark & \cmark & Tactical & \xmark & \xmark \\
\cite{korsvik2011large} & \cmark & \xmark & \cmark & Operational & \xmark & \xmark \\
\cite{gao2023tramp} & \cmark & \cmark & \cmark & Operational & \xmark & \xmark \\
\cite{song2025bulk} & \cmark & \cmark & \xmark & Tactical & \xmark & \xmark \\
\cite{wu2021robust} & \cmark & \xmark & \cmark & Operational & \cmark & \xmark \\
\cite{alvarez2011robust} & \cmark & \xmark & \cmark & Strategic & \cmark & \xmark \\
\cite{pantuso2016uncertainty} & \cmark & \xmark & \cmark & Strategic & \cmark & \xmark \\
\cite{arslan2017bulk} & \cmark & \xmark & \cmark & Tactical & \cmark & \xmark \\
\cite{de2023sailing} & \cmark & \cmark & \cmark & Operational & \cmark & IMO 2020 \\
\cite{sung2022multi} & \xmark & \cmark & \xmark & Operational & \xmark & GHG Reduction \\
\cite{wang2021voyage} & \cmark & \cmark & \xmark & Operational & \xmark & GHG Reduction \\
\cite{zhuge2024shipping} & \cmark & \cmark & \xmark & Operational & \xmark & IMO ECA \\
\cite{lagouvardou2022impacts} & \cmark & \cmark & \xmark & Operational & \xmark & MBMs \\
\cite{shang2024would} & \xmark & \xmark & \xmark & Strategic & \xmark & MBMs \\
\cite{wang2021paradox} & \xmark & \xmark & \xmark & Operational & \xmark & IMO CII \\
\cite{cheng2025cargo} & \cmark & \cmark & \cmark & Operational & \xmark & IMO CII \\
\textbf{This paper} & \cmark & \cmark & \cmark & Tactical & \cmark & IMO CII \\
\bottomrule
\end{tabularx}
}
\vspace{2mm}
\caption*{\raggedright \setstretch{1} \footnotesize
Note: For \textit{Planning Level}, ``Operational'' refers to short-term planning (weeks to months), such as port-to-port routing and voyage scheduling; ``Tactical'' refers to medium-term planning (months to a year), such as trade route design, regional service optimization, chartering, and other deployment decisions; ``Strategic'' refers to long-term planning (over one year), such as new ship construction, fleet renewal, or ship demolition and other firm-level investment decisions. For \textit{Environmental Policy}, ``IMO CII'' refers to the Carbon Intensity Indicator; ``IMO 2020'' refers to the global sulphur cap regulation; ``GHG Reduction'' refers to the goal of reducing greenhouse gas (GHG) emissions; ``IMO ECA'' refers to Emission Control Areas under MARPOL Annex VI; ``MBMs'' refers to Market-Based Measures, such as bunker levy and carbon pricing.}
\end{table}

We summarize the literature mentioned above in Table~\ref{tab:literature_compare}. The above literature has focused on the decision making in tramp shipping market, considering uncertain market conditions and environmental policies, and some have also made suggestions for such environmental policies. However, previous studies have not incorporate long-term uncertain tramp shipping market conditions and the constraints of CII policy into a single model, nor have there been extensive numerical experiments conducted to reveal the potential paradox in existing policies through such a model. Research such as \cite{cheng2025cargo} incorporates CII constraints only within a short-term deterministic framework, focusing primarily on operational design. While such models are essential for voyage-level efficiency, optimizing a single voyage in isolation may lead to locally optimal decisions that jeopardize annual CII compliance and fail to capture interactions among decisions across different stages of the planning horizon. Moreover, modeling long-term planning at such a detailed operational level would require high-resolution data—such as specific time windows for each cargo in every voyage throughout the year—which is often unavailable and impractical for shipping companies to obtain or manage. Therefore, a tactical-level model is necessary to bridge the gap between strategic environmental targets and operational execution. Tactical planning emphasizes focusing on regional routes connecting broader areas rather than individual ports, and on service frequency and capacity allocation over the long term, rather than on short-term, single-voyage operations. This could provide high-level deployment decisions that serve as the foundation and input for the operational optimization in studies such as \cite{cheng2025cargo}. Our research develops an enhanced tactical tramp fleet deployment model that integrates the annual CII regulation alongside long-term uncertainties including demand, fuel prices, and freight rates at a higher fleet deployment level. This model simultaneously addresses both environmental and economic considerations, revealing the paradox inherent in the CII policy and demonstrating the value of future information for shipping companies. Besides, based on the proposed model, we develop an algorithm that combines parameter search and solver techniques to improve solution efficiency, achieving an acceptable result in significantly less time compared to using the solver alone.

\section{Problem description}
\label{sec:problem_description}

\subsection{Tramp shipping market operation}
\label{sec:tramp_market}

In the tramp shipping market, the two primary stakeholders are the shipowner (carrier) and the shipper. Shipowners own and operate vessels, making decisions regarding cargo selection, routing, and speed to optimize revenue \citep{wu2021robust}. Shippers, on the other hand, are entities that require cargo transportation, negotiating with tramp shipping companies (shipowners) \citep{meng2015tailored}.

Ships in the market load goods at specific ports and unload them at other designated ports, based on the transportation demands \citep{staalhane2015branch, branchini2015routing}. These demands can typically be categorized into two types: contractual demand and spot market demand. Contractual demand refers to transportation needs defined by contracts between shippers and shipowners, outlining the total cargo volume to be transported over a specified period, the transportation frequency, the origin and destination ports, and the freight rates \citep{norstad2011tramp, christiansen2007maritime, bronmo2007multi}. Spot market demand, on the other hand, involves ad-hoc transportation needs announced by cargo owners that are not covered by long-term contracts. These are optional cargoes that shipping companies can choose to carry, depending on profitability and available fleet capacity \citep{norstad2011tramp}. For certain commodities, such as oil, the loading and unloading ports tend to be relatively stable. In such cases, transportation links connecting two geographical regions—comprising multiple loading ports in the origin area and multiple unloading ports in the destination area—are known as trade lanes \citep{gu2019can}.

\subsection{Tactical planning in tramp shipping deployment problems}
\label{sec:plan}

Shipowners need to allocate their vessels to different trade lanes to meet the cargo obligations specified in the each contract, and to transport spot cargoes based on the spot market demand and available capacity. Incorporating such decision-making process, this study focuses on long-term deployment problems, rather than short-term routing problems, considering future uncertain market conditions. Following the distinctions between tactical and operational planning in tramp shipping \citep{pache2020tactical}, this study adopts a tactical-level approach to capacity deployment. It focuses on regional-scale decisions (e.g., Asia Pacific, Middle East, South Africa) over a planning horizon of several months, but not exceeding one year, similar to the settings in \citet{wang2018planning}, \citet{alvarez2011robust}, and \citet{pantuso2016uncertainty}. This contrasts with operational-level models, which address voyage-specific planning and individual port calls over days or weeks, as in \citet{meng2015tailored}, \citet{korsvik2011large}, and \citet{cheng2025cargo}. Regions are modeled as nodes, with trade lanes as connections. Despite tramp shipping's route flexibility, stable crude oil production and consumption regions have established relatively stable trade routes \citep{prochazka2019contracting}. These routes are composed of trade lanes linked by ballast segments, and serve as the basis for route decisions, similar to \cite{pantuso2016uncertainty}. Routes range from single-lane round trips to multi-lane combinations. Demand and service are represented at a high level, with contracts aggregating demand along routes and accommodating uniform capacity types, representing transportation needs from multiple customers.

Deployment decisions entail the assignment of vessels to various trade lanes, under the constraints related to contracts (such as the volume of the cargo, the original and destination area, and the frequency of transportation for the cargo), spot market conditions (volume of cargoes available in spot market in a certain period and the original and destination area of a cargo), and ship capacities (the maximum volume of cargoes that can be carried on a ship). The process involves deciding the optimal routes for each ship, the transitions between different routes, and the frequency of operations each ship should perform on a given trade lane. These tactical capacity deployment choices directly dictate the operational sailing schedule from a long-run view. Ships also need to determine which type of cargo and the quantity to transport on the spot market to increase earnings, a decision that integrates the operational-level task of cargo selection into the tactical profit-maximization objective. Furthermore, selecting an appropriate vessel sailing speed is essential, as this key operational lever is optimized within the tactical model to balance voyage duration, fuel consumption, and emissions. This requires a balance between higher speeds, which increase revenue by enabling higher cargo transportation volume within a given time, and lower speeds, which reduce fuel costs and other expenses. Crucially, this is an operational-level decision which impacts the tactical year-long profits and carbon emissions. It is worth noting that shipowners in tramp shipping have a definite understanding of short-term market conditions, while long-term market conditions are more uncertain. When in short-term, where market conditions are certain, shipowners can make decisions based solely on the available information. However, for longer-term planning, where market information becomes uncertain, shipowners need to consider all possible scenarios. This approach ensures that they can meet contractual requirements and achieve higher profits when any of these scenarios actually realize in future. 

In summary, tactical planning in tramp shipping deployment involves a complex interplay of fulfilling contractual obligations, maximizing profits through spot market participation, and managing deployment decisions under varying levels of market uncertainty and the constraints of ship capacities. By integrating fundamental operational decisions into a long-term stochastic planning framework, our model bridges the tactical and operational levels, ensuring that the high-level plan faithfully reflects operational realities while accommodating long-term constraints and interactions.

\subsection{Introduction of CII to fleet deployment}
\label{sec:intro_CII_to_model}

This study incorporate the CII introduced by the IMO into fleet deployment problem, which add an additional layer of complexity to the basic problem. Based on \cite{imo2020fourth}, \cite{mepc2021336}, \cite{mepc2022352} and \cite{wang2021paradox}, we note two main formulations under discussion: the Supply-based CII and the Demand-based CII, defined as follows:
\begin{align}
    \text{Supply-based CII} &= \frac{M}{C \times D^T},\\
    \text{Demand-based CII} &= \frac{M}{\sum_{i} Q_i \times D^L_i}
    \label{eq:cii}
\end{align}
where $M$ represents the annual carbon emission (g) of the ship in a year, $C$ represents the deadweight tonnage (DWT) of the ship, $D^T$ represents the total distance traveled (in nautical miles) by the ship in a year, $Q_i$ represents the cargo volume (ton) carried during the $i_{th}$ laden voyage of the ship in a year, and $D^L_i$ represents the distance traveled (in nautical miles) during the $i_{th}$ laden voyage by the ship in a year. 

The fundamental difference lies in their denominators. Supply-Based CII is calculated based on the vessel's capacity and is suitable for formal ratings, for which data collection is relatively straightforward, but it may not accurately reflect actual transportation demands as well as Demand-Based CII, because it considers DWT rather than actual cargo volume transported. On the other hand, Demand-Based CII is calculated based on the actual cargo volume and more directly reflects actual transportation demands. However, it faces challenges in data collection and confidentiality, and is currently mainly used for trial assessments, according to \cite{mepc2022352}.

{
Critically, for the same vessel and annual emissions ($M$), the Supply-based CII denominator is always larger than or equal to its Demand-based counterpart. This relationship is proven by the following inequality chain:
\begin{align}
    \sum_{i} Q_i \times D_i^L \leq \max_i\{Q_i\} \times \sum_{i} D_i^L \leq C \times \sum_{i} D_i^L \leq C \times D^T.
    \label{neq:cii_denominator}
\end{align}

This mathematical relationship leads to a pivotal policy implication: if the same numerical CII threshold is applied to both formulations, the Demand-based CII imposes a stricter constraint. This is because it only recognizes laden voyages as productive work, resulting in a smaller denominator and thus a higher CII value for the same emissions. As the CII standard decreases, vessels under Demand-based CII reach the violation threshold sooner, whereas those under Supply-based CII can lower their CII by sailing in ballast, as shown in Figure \ref{fig:cii_change}. Specifically, since ballast voyages ($Q_i = 0$) contribute to emissions in the numerator $M$ but not to the denominator of the Demand-based CII, ship operators cannot improve their rating by optimizing ballast legs. In fact, the emissions from ballast sailing only worsen the CII score. In contrast, under the Supply-based CII, both emissions and distance from ballast sailing are accounted for in the numerator and denominator, respectively. Given that the carbon emission intensity during ballast voyages is typically lower than during laden voyages, increasing ballast distance can paradoxically help a vessel achieve a better Supply-based CII score by diluting the average emission intensity. This discrepancy creates a policy paradox: a stricter CII standard, while designed to reduce emissions, can incentivize vessels to undertake additional ballast sailing under the Supply-based CII, thereby increasing total carbon emissions. We provide a theoretical proof of the paradox for a single ship in Appendix~\ref{appendix:supply_cii_theory} to help readers better understand its mechanism. This paradox will be formally captured in our mathematical formulation in Section~\ref{sec:mathematical_formulation}, and its environmental implications will be revealed through the numerical analysis in Section~\ref{sec:compare_cii}.

\begin{figure}[!h]
    \centering
    \includegraphics[width=0.9\linewidth]{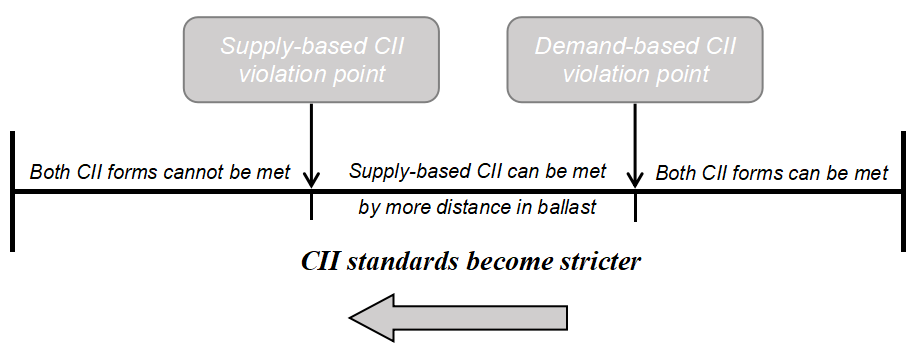}
    \caption{Results changing with CII standards}
    \label{fig:cii_change}
\end{figure}
}

From an annual perspective, each vessel needs to reduce its CII to below a given threshold. The introduction of CII presents challenges for existing fleet deployment because CII is an annual metric, and decisions made by vessels at any given time must be considered in conjunction with decisions made during other periods of the year to ensure that the CII remains below the threshold. This means that when a vessel makes a decision in the present, it must take into account all possible future uncertainties, making its current decision sufficiently conservative so that it can adjust decisions at future points in time to meet the annual CII threshold. To reflect the year-long constraints that CII imposes on fleet deployment and to capture the characteristic of short-term certainty and long-term uncertainty in market conditions, thereby exploring the impact of uncertain market information on fleet deployment, our study employ a two-stage stochastic programming model. The division into two stages is based on the degree of certainty of market information. In the first stage (the next few months), market information is fully known, while in the second stage (from the end of the first stage to the end of the year), market information is uncertain, including demand, freight rates, and fuel costs. For both stages, the constraints related to the market (such as demand volume and frequency) are independent of each other because we assume that the market information of each stage is unrelated to the other; however, the constraints related to CII are linked because CII is an annual metric that spans both the first and second stages.

In this study, we examine a shipping company that is making fleet deployment decisions for the remainder of the year. For a shipping company, fleet deployment planning may be required at any point during the year. Therefore, in our two-stage stochastic model, the initial time can be any point within the year. From this starting point, we account for the initial state of all vessels, including their current carbon emissions, work loads and routes they have been allocated to, and proceed with modeling the first and second stages. However, it is important to note that the later in the year the planning begins, the higher the risk of violating the CII. As year-end approaches, vessels are more likely to have high CII values, making it increasingly difficult to lower them to the required range through deployment adjustments.

\section{Model formulation}
\label{sec:model_formulation}

\subsection{Model setup}

In our deployment problem, planning begins at any point within a year, with the remaining time of that year serving as the planning period. This period is divided into two stages: the first stage (\textbf{P-1}) spans from the starting point to a few months later, but no later than year-end; the second stage (\textbf{P-2}) covers the period from the end of \textbf{P-1} to the year's conclusion. This division reflects differences in market information and confidence levels. For \textbf{P-1}, shipping companies have reliable information about the contract and spot markets. However, due to long-term market volatility, precise information for \textbf{P-2} is unavailable. This setting reflects real-world market fluctuations and planning uncertainty, which are not captured in the deterministic operational models proposed in previous studies, such as \cite{meng2015tailored}, \cite{korsvik2011large}, and \cite{cheng2025cargo}.

As described in Section \ref{sec:plan}, we define macro-level geographical areas as nodes and trade lanes as links connecting these nodes. This modeling approach differs fundamentally from the commonly studied port-to-port models incorporating time windows, which are typical in short-term operational planning literature like \cite{cheng2025cargo}. A route is then defined as a collection of specific trade lanes. Additionally, routes may include ballast segments connecting two trade lanes when the destination of one and the origin of the next are not geographically close. These ballast segments, where vessels travel without cargo, are an essential component of tramp shipping operations and are accounted for in our model. An illustration of this is shown in Figure \ref{fig:potential_route} \footnote{The world map here is sourced from the GeoPandas library in Python.}, where trade lane 1 (TR1) transports goods from North America to South America, trade lane 2 (TR2) transports goods from South America to Europe, and the two ballast sailing legs links the two trade lanes so that they can be combined into a route. Figure \ref{fig:abstract_route} presents an abstract representation of a route which is similar to Figure \ref{fig:potential_route}, consisting of TR1, TR2, and TR3, and following the sequence from TR1 to TR2 and then to TR3.

\begin{figure}[!h]
    \begin{subfigure}[b]{0.55\textwidth}
        \centering
        \includegraphics[width=\linewidth]{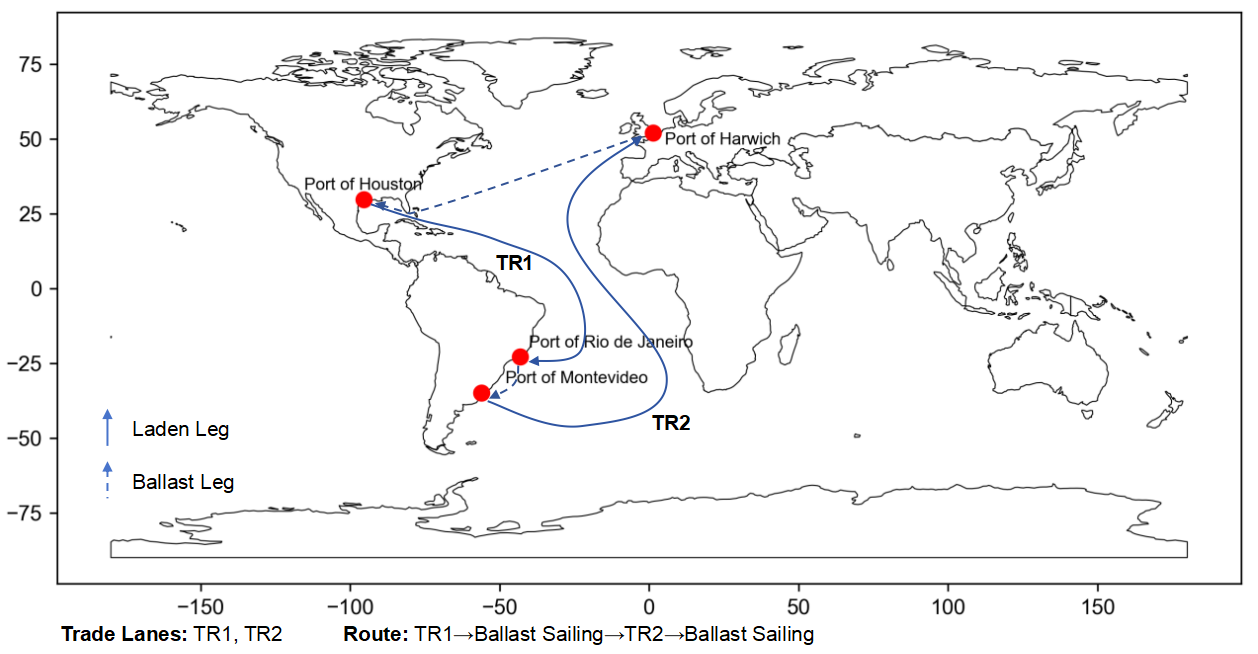}
        \caption{An example of potential route on world map}
        \label{fig:potential_route}
    \end{subfigure}%
    \hfill % Add some space between subfigures
    \begin{subfigure}[b]{0.45\textwidth}
        \centering
        \includegraphics[width=\linewidth]{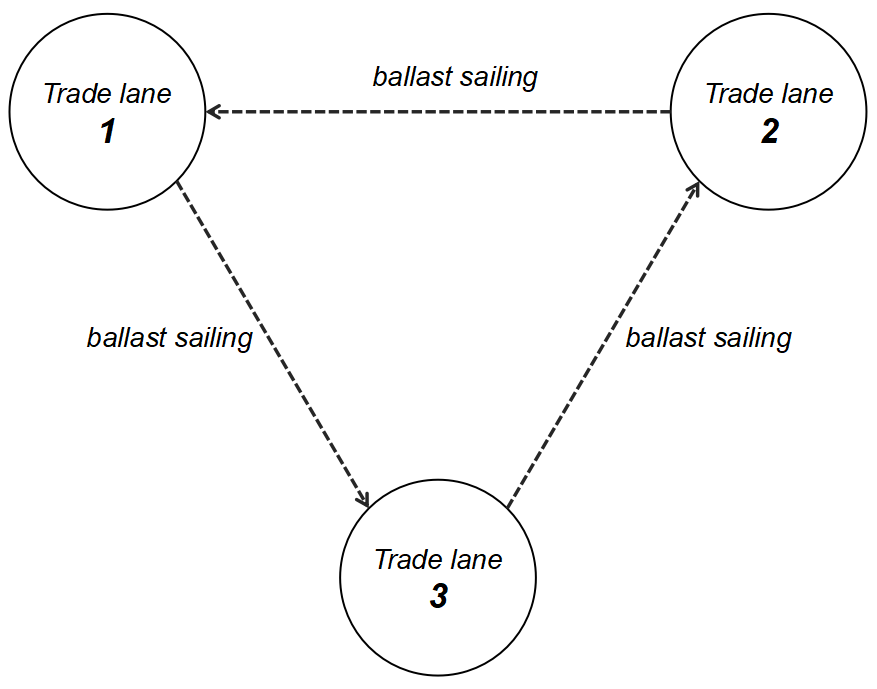}
        \caption{Abstract diagram of a route}
        \label{fig:abstract_route}
    \end{subfigure}
    \caption{Illustration of the deployment problem}
    \caption*{\raggedright \setstretch{1}Note: Figure \ref{fig:potential_route} shows a possible route consists of two trade lanes, one is from North America to South America, the other is from South America to Europe. The two trade lanes are linked by ballast sailing legs, so that ships can return to export areas from import areas. Figure \ref{fig:abstract_route} is an abstract represent of a possible route consists of three trade lanes and three ballast sailing legs to link the trade lanes.}
\end{figure}

\begin{table}[!h]
    \centering
    \renewcommand{\arraystretch}{1.2}
    \begin{tabular}{c|*{5}{m{0.8cm}}|*{6}{m{0.8cm}}}
    \hline
    \multirow{2}{*}{\textbf{Routes}} & \multicolumn{5}{c|}{Select the First Route in \textbf{P-1}} & \multicolumn{5}{c}{Route Transition in \textbf{P-1}} \\ \cline{2-11} 
    & \textbf{1} &\textbf{2} &\textbf{3} &\textbf{4} &
    \textbf{5} & \textbf{1} & \textbf{2} & \textbf{3} & \textbf{4} & \textbf{5} \\ 
    \hline
    \textbf{1}                       & 0          & 0          & 0          & 0          & 0          & 0          & 0          & 0          & 1          & 0\\ \hline
    \textbf{2}                       & 0          & 0          & 0          & 0          & 0          & 0          & 0          & 0          & 0          & 0\\ \hline
    \textbf{3}                       & 1          & 0          & 0          & 0          & 0          & 0          & 0          & 0          & 0          & 0\\ \hline
    \textbf{4}                       & 0          & 0          & 0          & 0          & 0          & 1          & 0          & 0          & 0          & 1\\ \hline
    \textbf{5}                       & 0          & 0          & 0          & 0          & 0          & 0          & 0          & 0          & 1          & 0\\ \hline
    \end{tabular}
    \caption{Route Transition Matrix in \textbf{P-1}}
    \caption*{\raggedright \setstretch{1}Note: This table shows the route transition matrix in \textbf{P-1} (the first planning stage). The first column of rows represent the original routes (before transit decisions), and the second row of columns represent the routes selected after the transit decision. A value of \textbf{1} in a grid (i, j) means the ship will transition from route \textbf{i} to route \textbf{j}. For example, in the grid (3,1), the value 1 indicates that the ship will transition from route 3 to route 1. This matrix helps in understanding how routes change based on ship transit decisions.}
    \label{tab:p1 route transition}
\end{table}

\begin{table}[!h]
    \centering
    \renewcommand{\arraystretch}{1.2}
    \begin{tabular}{c|*{5}{m{0.8cm}}|*{6}{m{0.8cm}}}
    \hline
    \multirow{2}{*}{\textbf{Routes}} & \multicolumn{5}{c|}{Select the First Route in \textbf{P-2}} & \multicolumn{5}{c}{Route Transition in \textbf{P-2}} \\ \cline{2-11} 
    & \textbf{1} &\textbf{2} &\textbf{3} &\textbf{4} &
    \textbf{5} & \textbf{1} & \textbf{2} & \textbf{3} & \textbf{4} & \textbf{5} \\ 
    \hline
    \textbf{1}                       & 0          & 0          & 0          & 0          & 0          & 0          & 0          & 0          & 0          & 0\\ \hline
    \textbf{2}                       & 0          & 0          & 0          & 0          & 0          & 0          & 0          & 1          & 1          & 0\\ \hline
    \textbf{3}                       & 0          & 0          & 0          & 0          & 0          & 0          & 1          & 0          & 0          & 0\\ \hline
    \textbf{4}                       & 0          & 0          & 0          & 0          & 0          & 0          & 1          & 0          & 0          & 0\\ \hline
    \textbf{5}                       & 0          & 0          & 0          & 1          & 0          & 0          & 0          & 0          & 0          & 0\\ \hline
    \end{tabular}
    \caption{Route Transition Matrix in \textbf{P-2}}
    \caption*{\raggedright \setstretch{1}Note: See notes in Table \ref{tab:p1 route transition}. \textbf{P-2} denotes the second planning stage.}
    \label{tab:p2 route transition}
\end{table}

To model vessel route selection and transitions, we introduce binary variables to indicate whether a vessel operates on a specific route during a given stage or scenario, or switches between routes within or across periods. Notably, we assume that each vessel completes an integer number of full cycles. Additionally, the duration of route switching is represented as fixed parameters, reflecting the distance between the starting points of two routes.

During \textbf{P-1} and \textbf{P-2}, as well as the transition between the two, route switching follows a specific sequence. Within each stage, ships complete their assigned cargo transportation tasks along a route at a set speed before switching to another route. This approach minimizes costs and carbon emissions. Conversely, leaving a route mid-task and returning later within the same stage would increase both costs and emissions.

To illustrate this, consider the example in Tables \ref{tab:p1 route transition} and \ref{tab:p2 route transition}, where there are five routes. Suppose a ship starts on route 3 at the beginning of the planning period. During \textbf{P-1}, it transitions sequentially to routes 1, 4, and 5, completing tasks on each before entering \textbf{P-2}. In \textbf{P-2}, it switches to routes 4, 2, and 3 in succession to perform its transportation tasks. The route transition process is depicted in Figure \ref{fig:route_change}, which is more illustrative.

\begin{figure}[!h]
    \centering
    \includegraphics[width=0.75\linewidth]{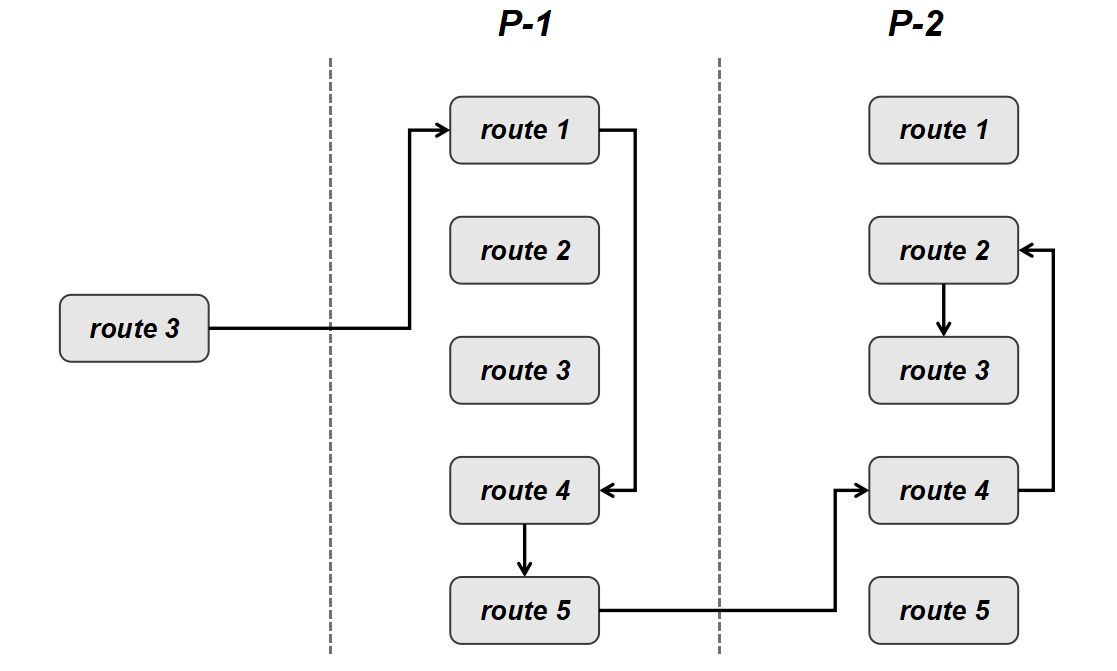}
    \caption{Route transition during \textbf{P-1} and 
    \textbf{P-2}}
    \caption*{\raggedright \setstretch{1}Note: This figure illustrate the transit decision process described in Table \ref{tab:p1 route transition} (the first planning stage \textbf{P-1}) and Table \ref{tab:p2 route transition} (the second planning stage \textbf{P-2}).}
    \label{fig:route_change}
\end{figure}

Both \textbf{P-1} and \textbf{P-2} involve two processes: transitioning from the initial route to the first selected route and the sequence of subsequent route transitions. These processes are represented by two matrices. The first, an asymmetric matrix $A$, has a $1$ in the $i$-th row and $j$-th column to indicate a transition from the initial route $i$ to the first selected route $j$. $i$ and $j$ may be equal if the initial route is the same as the first task route. In $A$, only one element is $1$, with all others being $0$.

The second, a symmetric matrix $D$ with $n$ rows and $n$ columns, represents subsequent route transitions. Let $D_{ij}$ denote the element in the $i$-th row and $j$-th column of the transition matrix $D$. A value of $D_{ij} = 1$ (and symmetrically $D_{ji} = 1$) indicates a transition between routes $i$ and $j$. The diagonal elements are set to zero, as transitions represent changes between distinct routes. The sequence of transitions is not directly indicated by row and column positions but is constrained by connections between nodes. According to equation \ref{eq:matrix1}, the numerical relationships ensure that each matrix uniquely represents one sequence of route transitions and vice versa, based on tree properties in graph theory:
\begin{align}
    \sum_{i=1}^{n}\sum_{j=1}^{n} D_{ij} = 2 \times (m-1),\label{eq:matrix1}
\end{align}
where $D_{ij}$ is the element in the $i$-th row and $j$-th column of matrix $D$, and $m$ is the total number of routes selected in the stage by the ship.

Given the uncertainties in contract demands, spot market demands, fuel prices, and freight rates, these factors are modeled in \textbf{P-2} as stochastic parameters following specific probability distributions. They are incorporated into the constraints of \textbf{P-2}, influencing solutions for both \textbf{P-1} and \textbf{P-2}. For CII, the constraints apply across both stages since CII is calculated based on annual carbon emissions relative to workload intensity \citep{mepc2022352}. Thus, the CII constraints must be satisfied regardless of the \textbf{P-2} scenario.

In both \textbf{P-1} and each scenario in \textbf{P-2}, shipowners must decide for each vessel its route selection and route transitions, speed choices, the number of voyages to operate on each selected route, which contract and spot market cargoes to carry, and the loading volume for each cargo type. When no suitable cargo is available, shipowners must also decide whether a vessel should remain idle at port or perform ballast sailing, which directly affects compliance with the annual CII constraint.

The two stages are tightly interrelated through two types of constraints: (1) the annual CII constraint, which couples the carbon intensity performance of the entire planning horizon, and (2) the route transition constraint, which requires each vessel at the beginning of \textbf{P-2} to inherit its route position at the end of \textbf{P-1}. These constraints ensure that first-stage decisions directly restrict the feasible set and cost performance of second-stage decisions, and conversely, expected conditions in the second stage influence optimal first-stage planning.

\subsection{Mathematical formulation}
\label{sec:mathematical_formulation}

In this section, we start by defining notations of sets, parameters, and decision variables to mathematically formulate the model described in Section \ref{sec:model_formulation}, as shown in Table \ref{tab:notations}. A two-stage stochastic mixed integer linear programming (MILP) model is then provided to form the problem.

{\small
\begin{longtable}[htbp]{@{}l p{7.5cm} p{2.5cm}@{}}
\caption{Notations of sets, parameters and decision variables} \\
\toprule
\textbf{Symbol} & \textbf{Description} & \textbf{Unit} \\
\midrule
\endfirsthead
\multicolumn{3}{@{}l@{}}{Notations of sets, parameters and decision variables (continued)} \\
\toprule
\endhead

\textbf{Sets} & & \\
\midrule
$ V $ & Set of ships; $v$. & -- \\
$ N $ & Set of trade lanes; $i$. & -- \\
$ R $ & Set of trade routes; $r$. & -- \\
$ K $ & Set of capacity types; $k$. & -- \\
$ C $ & Set of contracts; $c$. & -- \\
$S$ & Set of scenarios; $s$. & -- \\
$\mathcal{E}_v$ & Set of speed alternatives for ship $v$; $e$. & -- \\
$ R_v \subseteq R $ & The set of routes that can be sailed by ship $v$; $r$. & -- \\
$ R_{iv} \subseteq R $ & The set of routes servicing trade lane $i$ that can be sailed by ship $v$; $r$. & -- \\
$C_{i}^{TR} \subseteq C$ & The set of contracts serviced by trade lane $i$; $c$. & -- \\
$V_i \subseteq V$ & The set of ships that can sail trade lane $i$; $v$. & -- \\
$C_{k}^{CP} \subseteq C$ & The set of contracts compatible with capacity type $k$; $c$. & -- \\
$K_c \subseteq K$ & The set of capacity types compatible with contract $c$; $k$. & -- \\

\midrule
\textbf{CII Formulation Parameters} & & \\
\midrule
$ M $ & Annual carbon emission of a ship in a year (used to introduce the CII calculation). & gram (g) \\
$ C $ & Deadweight tonnage of a ship (used to introduce the CII calculation). & tonne (t) \\
$ Q_i $ & The cargo volume carried by a ship during the $i_{th}$ laden voyage (used to introduce the CII calculation). & tonne (t) \\
$ D^T $ & Total distance traveled by a ship in a year (used to introduce the CII calculation). & nautical mile \\
$ D_i^L $ & The distance traveled during the $i_{th}$ laden voyage by a ship (used to introduce the CII calculation). & nautical mile \\

\midrule
\textbf{Deterministic Parameters} & & \\
\midrule
$ M_{v}^{1}, M_{v}^{2} $ & Number of available service days for ship $v \in V$ during \textbf{P-1} and \textbf{P-2}, respectively. & day \\
$ Q_{vk} $ & Volume of capacity type $k\in K$ installed on ship $v\in V$. & tonne (t) \\
$ T_{vre} $ & Total travel period for ship $v\in V$ to complete a round trip on route $r \in R_v$ with speed alternative $e\in \mathcal{E}_v$. & day \\
\( T_{vrr'e} \) & Total travel period for ship \( v\in V \) to complete a transfer trip from route \( r \in R_v \) to route \( r'\in R_v \) with speed \( e \in \mathcal{E}_v \). & day \\
\( F_c^{1}, F_c^{2} \) & Stochastic frequency requirement of contract \( c \in C \) during \textbf{P-1} and \textbf{P-2}, respectively. & voyage \\
\( D_c \) & Deterministic demand of contract \( c \in C \) during \textbf{P-1}. & tonne (t) \\
\( D_{ik}^{SP} \) & Volume of spot cargo available on trade lane \( i\in N $ compatible with capacity type $ k \in K $ during \textbf{P-1}. & tonne (t) \\
\( C^{RT}_{vre} \) & Cost for ship $v\in V$ to complete a round trip on route $ r \in R_v $ with speed alternative $ e \in \mathcal{E}_v $ during \textbf{P-1}. & USD (\$) \\
\( C^{TF}_{vrr'e} \) & Cost for ship $ v \in V $ to transfer route from $r \in R_v$ to $r'\in R_v$ with speed alternative $e \in \mathcal{E}_v$ during \textbf{P-1}. & USD (\$) \\
\( C_{v}^{ballast} \) & Cost for ship $v\in V$ to ballast sail per day during \textbf{P-1}. & USD (\$)/day \\
\( C_{v}^{port} \) & Cost for ship $v \in V$ at port per day during \textbf{P-1}. & USD (\$)/day \\
\( R_{ik}^{SP} \) & Revenue of delivering one ton of spot cargo with capacity type \( k \in K \) on trade lane \( i\in N \) during \textbf{P-1}. & USD (\$)/t \\
\( E_{v}^{port} \) & Carbon emission of ship $v$ at port per day. & gram (g)/day \\
\( E_{vre} \) & Carbon emission of ship $v\in V$ to complete a round trip on route \( r\in R_v \) with speed alternative \( e \in \mathcal{E}_v \). & gram (g) \\
\( E_{vrr'e} \) & Carbon emission of ship $v\in V$ to complete a transfer trip from route \( r \in R_v \) to route \( r' \in R_v \) with speed \( e \in \mathcal{E}_v \). & gram (g) \\
\( E_{v}^{ballast} \) & Carbon emission of ship $v\in V$ ballast sailing per day. & gram (g)/day \\
\( E_{v}^{port} \) & Carbon emission of ship $v\in V$ at port per day. & gram (g)/day \\
\( E_{v}^{0} \) & Carbon emission of ship $v\in V$ before planning. & gram (g) \\
\( W_{v}^{0} \) & Working intensity (laden sailing distance $\times$ corresponding loads) of ship $v\in V$ before planning. & t$\cdot$n mile \\
\( L_{i} \) & Laden sailing distance of trade lane $i\in N$. & nautical mile \\
\( L_{v}^{ballast} \) & Distance of ship $v\in V$ per day when ballast sailing. & nautical mile/day \\
\( L_{r}^{route} \) & Distance of route $r\in R$. & nautical mile \\
\( L_{rr'} \) & Distance of switching from route $r\in R$ to $r'\in R$. & nautical mile \\
\(y_{vr}^{1,0}\) & 0-1 variable, whether ship $v\in V$ on route $r\in R$ before planning. & -- \\
\(CII_{v}\) & The CII standard of ship $v$. & g/(t$\cdot$nautical mile) \\

\midrule
\textbf{Stochastic Parameters} & & \\
\midrule
$ p^s $ & The probability of scenarios $s\in S$ taking place during \textbf{P-2}. & -- \\
$ D_{cs} $ & Stochastic demand of contract $c\in C$ during \textbf{P-2} in scenario $s\in S$. & tonne (t) \\
\( D_{iks}^{SP} \) & Volume of spot cargo available on trade lane \( i\in N \) compatible with capacity type \( k\in K \) during \textbf{P-2} in scenario $s\in S$. & tonne (t) \\
\( C^{RT}_{vres} \) & Cost for ship \( v\in V \) to complete a round trip on route \( r\in R_v \) with speed alternative \( e\in \mathcal{E}_v \) during \textbf{P-2} in scenario $s$. & USD (\$) \\
\( C^{TF}_{vrr'es} \) & Cost for ship \( v\in V \) to transfer route from $r\in R_v$ to $r' \in R_v$ with speed alternative $e$ during \textbf{P-2} in scenario $s\in S$. & USD (\$) \\
\( C_{vs}^{ballast} \) & Cost for ship $v\in V$ to ballast sail per day during \textbf{P-2} in scenario $s\in S$. & USD (\$)/day \\
\( C_{vs}^{port} \) & Cost for ship $v\in V$ at port per day during \textbf{P-2} in scenario $s\in S$. & USD (\$)/day \\
\( R_{iks}^{SP} \) & Revenue of delivering one ton of spot cargo with capacity type \( k\in K \) on trade lane \( i\in N \) during \textbf{P-2} in scenario $s\in S$. & USD (\$)/t \\

\midrule
\textbf{Decision Variables} & & \\
\midrule
$ x_{vre}, x_{vres} $ & Integer variable, number of round trips sailed by ship $v\in V$ on route $r\in R_v$ with speed $e\in \mathcal{E}_v$ during P-1 and on route $r$ during P-2 in scenario $s\in S$. & voyage \\
$ y_{vr}^1, y_{vrs}^2 $ & Binary variable, whether ship $v\in V$ sails on route $r\in R_v$ during P-1 and on $r$ during P-2 in scenario $s\in S$. & -- \\
\( y_{vrs}^{2,0} \) & Binary variable, whether ship \( v\in V \) sails on route $r\in R_v$ at the beginning of scenarios $s\in S$. & -- \\
\( y_{vrr'e}^{1,1}, y_{vrr'es}^{2,1} \) & Binary variable, if ship \( v\in V \) is transferred from route \( r\in R_v \) to \( r'\in R_v \) with speed \( e \in\mathcal{E}_v$ at the beginning of \textbf{P-1}, \( y_{vrr'e}^{1,1} = 1 \), otherwise \( y_{vrr'e}^{1,1} = 0 \). The same meaning for $y_{vrr'es}^{2,1}$ at the beginning of \textbf{P-2}. & -- \\
\( y_{vrr'e}^{1,2}, y_{vrr'es}^{2,2} \) & Binary variable, if ship \( v\in V \) is transferred between route \( r\in R_v \) and $ r'\in R_v $ with speed \( e\in\mathcal{E}_v \) during \textbf{P-1}, \( y_{vrr'e}^{1,2} = 1 \), otherwise \( y_{vrr'e}^{1,2} = 0 \). The same meaning for $y_{vrr'es}^{2,2}$ during \textbf{P-2}. & -- \\
\( q^{C}_{ivkc}, q^{C}_{ivkcs} \) & Volume of contract \( c\in C \) carried by capacity type \( k\in K_c \) installed on ship \( v\in V \) on trade lane \( i\in N \) in \textbf{P-1} and in scenarios $s\in S$. & tonne (t) \\
\( q_{ivk}^{SP}, q_{ivks}^{SP} \) & Volume of spot cargo carried by capacity type \( k\in K \) installed on ship $ v\in V $ on trade lane $ i\in N $ in \textbf{P-1} and in scenarios $s\in S$. & tonne (t) \\
\( T_{v}^{1,port}, T_{vs}^{2,port} \) & Time spent in port by ship $v\in V$ other than performing tasks during \textbf{P-1} and during \textbf{P-2} in scenario $s\in S$. & day \\
\( T_{v}^{1,ballast}, T_{vs}^{2,ballast} \) & Time spent on ballast sailing by ship $v\in V$ other than performing tasks during \textbf{P-1} and during \textbf{P-2} in scenario $s\in S$. & day \\
\bottomrule
\label{tab:notations}
\end{longtable}}

\begin{align}
    \min \quad & \sum_{v \in V} \sum_{r \in R_v} \sum_{e \in \mathcal{E}_v} C^{RT}_{vre} x_{vre} + \sum_{v \in V} \sum_{r \in R_v} \sum_{r' \in R_v} \sum_{e \in \mathcal{E}_v} C^{TF}_{vrr'e} y_{vrr'e}^{1,1} + \sum_{v \in V} \sum_{r \in R_v} \sum_{r' \in R_v \setminus \{r\} } \sum_{e \in \mathcal{E}_v} \frac{1}{2} C^{TF}_{vrr'e} y_{vrr'e}^{1,2} \notag \\
    & + \sum_{v \in V} (C_{v}^{ballast}\cdot T_{v}^{1,ballast} + C_{v}^{port}\cdot T_{v}^{1,port}) - \sum_{i \in N} \sum_{v \in V_i} \sum_{k \in K} R_{ik}^{SP} q_{ivk}^{SP} \notag \\
    & + \sum_{s \in S} p^s [ \sum_{v \in V} \sum_{r \in R_v} \sum_{e \in \mathcal{E}_v} C^{RT}_{vres} x_{vres} + \sum_{v \in V} \sum_{r \in R_v} \sum_{r' \in R_v} \sum_{e \in \mathcal{E}_v} C^{TF}_{vrr'es} y_{vrr'es}^{2,1} \notag \\ & + \sum_{v \in V} \sum_{r \in R_v} \sum_{r' \in R_v \setminus \{r\}} \sum_{e \in \mathcal{E}_v} \frac{1}{2} C^{TF}_{vrr'es} y_{vrr'es}^{2,2} \notag \\
    & + \sum_{v \in V} (C_{vs}^{ballast}\cdot T_{vs}^{1,ballast} + C_{vs}^{port}\cdot T_{vs}^{1,port}) - \sum_{i \in N} \sum_{v \in V_i} \sum_{k \in K} R_{iks}^{SP} q_{ivks}^{SP}]
    \label{eq:obj}
\end{align}
\begin{align}
    \sum_{e \in \mathcal{E}_v}x_{vre} &\leq \mathcal{M} y_{vr}^1 & \forall v \in V, r \in R_v \label{eq:route_select_1}
\end{align}
\begin{align}
    \sum_{e \in \mathcal{E}_v} x_{vres} &\leq \mathcal{M} y_{vrs}^2 & \forall v \in V, r \in R_v, s \in S \label{eq:route_select_2}
\end{align}
\begin{align}
    y_{vr}^{1,0} &= \sum_{r' \in R_v} \sum_{e \in \mathcal{E}_v} y_{vrr'e}^{1,1}  & \forall v \in V, r \in R_v \label{eq:route_select_3}
\end{align}
\begin{align}
   y_{vr'}^1 &\geq \sum_{r \in R_v} \sum_{e \in \mathcal{E}_v}  y_{vrr'e}^{1,1} & \forall v \in V, r' \in R_v \label{eq:route_select_4}
\end{align}
\begin{align}
    \sum_{e \in \mathcal{E}_v} y_{vrr'e}^{1,2} &\leq 1 - (y_{vr}^1 - y_{vr'}^1) & \forall v \in V, r \in R_v, r' \in R_v\setminus \{r\} \label{eq:route_select_5}
\end{align}
\begin{align}
   y_{vrr'e}^{1,2} &= y_{vr're}^{1,2} & \forall v \in V, r \in R_v, r' \in R_v\setminus \{r\}, e \in \mathcal{E}_v \label{eq:route_select_6}
\end{align}
\begin{align}
    \sum_{r \in R_v} \sum_{r' \in R_v} \sum_{r'' \in R_v\setminus \{r'\}} \sum_{e \in \mathcal{E}_v} (y_{vrr'e}^{1,1} + \frac{1}{2} y_{vr'r''e}^{1,2}) &= \sum_{r \in R_v} y_{vr}^1 & \forall v \in V \label{eq:route_select_7}
\end{align}
\begin{align}
    y_{vr's}^{2,0} &\leq 2 - \sum_{r \in R_v} \sum_{r'' \in R_v \setminus \{r'\}} \sum_{e \in \mathcal{E}_v} (y_{vrr'e}^{1,1} + y_{vr'r''e}^{1,2}) & \forall v \in V, r' \in R_v, s \in S \label{eq:route_select_8}
\end{align}
\begin{align}
    y_{vr's}^{2,0} &\leq \sum_{r \in R_v} \sum_{r'' \in R_v \setminus \{r'\}} \sum_{e \in \mathcal{E}_v} (y_{vrr'e}^{1,1} + y_{vr'r''e}^{1,2}) & \forall v \in V, r' \in R_v, s \in S \label{eq:route_select_9}
\end{align}
\begin{align}
    \sum_{r \in R_v} y_{vrs}^{2,0} &= 1 & \forall v \in V, s \in S \label{eq:route_select_10}
\end{align}
\begin{align}
    y_{vrs}^{2,0} &= \sum_{r' \in R_v} \sum_{e \in \mathcal{E}_v} y_{vrr'es}^{2,1}  & \forall v \in V, r \in R_v, s \in S \label{eq:route_select_11}
\end{align}
\begin{align}
   y_{vr's}^2 &\geq \sum_{r \in R_v} \sum_{e \in \mathcal{E}_v}  y_{vrr'es}^{2,1} & \forall v \in V, r' \in R_v, s \in S \label{eq:route_select_12}
\end{align}
\begin{align}
    \sum_{e \in \mathcal{E}_v} y_{vrr'es}^{2,2} &\leq 1 - (y_{vrs}^2 - y_{vr's}^2) & \forall v \in V, r \in R_v, r' \in R_v\setminus \{r\}, s \in S \label{eq:route_select_13}
\end{align}
\begin{align}
    \sum_{r \in R_v} \sum_{r' \in R_v} \sum_{r'' \in R_v\setminus \{r'\}} \sum_{e \in \mathcal{E}_v} (y_{vrr'es}^{2,1} + \frac{1}{2} y_{vr'r''es}^{2,2}) &= \sum_{r \in R_v} y_{vrs}^2 & \forall v \in V, s \in S \label{eq:route_select_14}
\end{align}
\begin{align}
   y_{vrr'es}^{2,2} &= y_{vr'res}^{2,2} & \forall v \in V, r \in R_v, r' \in R_v\setminus \{r\}, e \in \mathcal{E}_v \label{eq:route_select_15}
\end{align}
\begin{align}
    \sum_{r \in R_{v}} \sum_{e \in \mathcal{E}_{v}} T_{vre} x_{vre} + \sum_{r \in R_{v}} \sum_{r' \in R_{v}} \sum_{e \in \mathcal{E}_{v}} T_{vrr'e} y_{vrr'e}^{1,1} + & \sum_{r \in R_{v}} \sum_{r' \in R_{v} \setminus \{r\} } \sum_{e \in \mathcal{E}_{v}} \frac{1}{2} T_{vrr'e} y_{vrr'e}^{1,2} \notag \\ & + T_{v}^{1,ballast} + T_{v}^{1,port} = M_{v}^{1} & \forall v \in V \label{eq:service-days_1}
\end{align}
\begin{align}
    \sum_{r \in R_{v}} \sum_{e \in \mathcal{E}_{v}} T_{vre} x_{vres} + \sum_{r \in R_{v}} \sum_{r' \in R_{v}} \sum_{e \in \mathcal{E}_{v}} T_{vrr'e} y_{vrr'es}^{2,1} + & \sum_{r \in R_{v}} \sum_{r' \in R_{v}} \sum_{e \in \mathcal{E}_{v}} \frac{1}{2} T_{vrr'e} y_{vrr'es}^{2,2} \notag \\ & + T_{vs}^{2,ballast} + T_{vs}^{2,port} = M_{v}^{2} & \forall v \in V, s \in S \label{eq:service-days_2}
\end{align}
\begin{align}
    \sum_{v \in V_i} \sum_{r \in R_{iv}} \sum_{e \in \mathcal{E}_{v}} x_{vre} &\geq F_c^{1} & \forall i \in N, \forall c \in C_{i}^{TR} \label{eq:contract-frequency_1}
\end{align}
\begin{align}
    \sum_{v \in V_i} \sum_{r \in R_{iv}} \sum_{e \in \mathcal{E}_{v}} x_{vres} &\geq F_c^{2} & \forall i \in N, \forall c \in C_{i}^{TR}, s \in S \label{eq:contract-frequency_2}
\end{align}
\begin{align}
    \sum_{v \in V_i} \sum_{k \in K_c} q^{C}_{ivkc} &= D_c & \forall i \in N, c \in C_{i}^{TR} \label{eq:contract-demand_1}
\end{align}
\begin{align}
    \sum_{v \in V_i} \sum_{k \in K_c} q^{C}_{ivkcs} &= D_{cs} & \forall i \in N, c \in C_{i}^{TR}, s \in S \label{eq:contract-demand_2}
\end{align}
\begin{align}
    \sum_{r \in R_{iv}} \sum_{e \in \mathcal{E}_v}  Q_{vk} x_{vre} &\geq \sum_{c \in C^{TR}_{i} \cap C^{CP}_k} q^{C}_{ivkc} + q_{ivk}^{SP} & \forall i \in N, v \in V_i, k \in K \label{eq:capacity_1}
\end{align}
\begin{align}
    \sum_{r \in R_{iv}} \sum_{e \in \mathcal{E}_v} Q_{vk} x_{vres} &\geq \sum_{c \in C^{TR}_{i} \cap C^{CP}_k} q^{C}_{ivkcs} + q_{ivks}^{SP} & \forall i \in N, v \in V_i, k \in K \label{eq:capacity_2}
\end{align}
\begin{align}
    \sum_{v \in V_i} q_{ivk}^{SP} &\leq D_{ik}^{SP} & \forall i \in N, k \in K \label{eq:spot-cargo_1}
\end{align}
\begin{align}
    \sum_{v \in V_i} q_{ivks}^{SP} &\leq D_{iks}^{SP} & \forall i \in N, k \in K, s \in S \label{eq:spot-cargo_2}
\end{align}
\begin{align}
    & E_{v}^{0} + \sum_{r \in R_{v}} \sum_{r' \in R_{v}} \sum_{r'' \in R_{v} \setminus \{r'\}} \sum_{e \in \mathcal{E}_v} [E_{vr'e} \cdot (x_{vr'e} + x_{vr'es}) + E_{vrr'e} \cdot (y_{vrr'e}^{1,1} + y_{vrr'es}^{2,1}) \notag \\ & + \frac{1}{2} E_{vr'r''e} \cdot (y_{vr'r''e}^{1,2} + y_{vr'r''es}^{2,2})] + E_{v}^{ballast} \cdot (T_{v}^{1,ballast} + T_{v}^{2,ballast}) + E_{v}^{port} \cdot (T_{v}^{1,port} + T_{vs}^{2,port}) \notag \\ & \leq CII_{v}  \cdot [W_{v}^{0} + \sum_{i \in N} \sum_{k \in K} L_{i} (\sum_{c \in C^{TR}_{i} \cap C^{CP}_k} q^{C}_{ivkc} + q_{ivk}^{SP} + \sum_{c \in C^{TR}_{i} \cap C^{CP}_k} q^{C}_{ivkcs} + q_{ivks}^{SP})] \ \ \forall v \in V, s \in S \label{eq:demand_cii}
\end{align}
\begin{align}
    & E_{v}^{0} + \sum_{r \in R_{v}} \sum_{r' \in R_{v}} \sum_{r'' \in R_{v} \setminus \{r'\}} \sum_{e \in \mathcal{E}_v} [E_{vr'e} \cdot (x_{vr'e} + x_{vr'es}) + E_{vrr'e} \cdot (y_{vrr'e}^{1,1} + y_{vrr'es}^{2,1}) \notag \\ & + \frac{1}{2} E_{vr'r''e} \cdot (y_{vr'r''e}^{1,2} + y_{vr'r''es}^{2,2})] + E_{v}^{ballast} \cdot (T_{v}^{1,ballast} + T_{vs}^{2,ballast}) + E_{v}^{port} \cdot (T_{v}^{1,port} + T_{vs}^{2,port}) \notag \\ & \leq CII_{v}  \cdot \{W_{v}^{0} + [\sum_{r \in R} L_{r}^{route}\cdot (\sum_{e \in \mathcal{E}_v}x_{vre}+ \sum_{e \in \mathcal{E}_v}x_{vres}) + \sum_{r \in R_{v}} \sum_{r' \in R_{v} \setminus \{r\}} \sum_{e \in \mathcal{E}_v} L_{rr'} \cdot (\frac{1}{2} y_{vrr'e}^{1,2} + \frac{1}{2} y_{vrr'es}^{2,2} \notag \\ & + y_{vrr'e}^{1,1} + y_{vrr'es}^{2,1}) + L_{v}^{ballast}\cdot (T_{v}^{1,ballast} + T_{vs}^{2,ballast})]\cdot \sum_{k \in K} Q_{vk}\} \ \ \ \forall v \in V, s \in S \label{eq:supply_cii}
\end{align}
\begin{align}
    x_{vre}, x_{vres} &\in \mathbb{N} & \forall v \in V, r \in R_v, s \in S
\label{eq:variable_1}
\end{align}
\begin{align} 
    y_{vr}^1, y_{vrs}^2, y_{vrr'e}^{1,1}, y_{vrr'es}^{2,1}, y_{vrr'e}^{1,2}, y_{vrr'es}^{2,2} &\in \{0, 1\} & \forall v \in V, r \in R_v, r' \in R_v, e\in \mathcal{E}_v, s \in S
    \label{eq:variable_2}
\end{align}
\begin{align}
    q_{ivkc}^{C}, q_{ivkcs}^{C} &\geq 0  & \forall i\in N, v \in V_i, k \in K, c\in C_i^{TR}\cap C_k^{CP}, s \in S \label{eq:variable_3}
\end{align}
\begin{align}
    q_{ivk}^{SP}, q_{ivks}^{SP} &\geq 0  & \forall i\in N, v \in V_i, k \in K, s \in S \label{eq:variable_4}
\end{align}
\begin{align}
    T_{v}^{1,ballast}, T_{vs}^{2,ballast}, T_{v}^{1,port}, T_{vs}^{2,port} &\geq 0  & \forall v \in V, s \in S \label{eq:variable_5}
\end{align}

In the above model, the objective function \eqref{eq:obj} minimizes the net cost, defined as the total cost of shipping operations minus the revenue from transporting spot market cargoes. The cost components include fleet operation and fuel expenses, with parameters $C^{RT}$, $C^{TF}$, $C^{ballast}$, and $C^{port}$, corresponding to the costs incurred during transportation, route switching, ballast sailing, and staying at port, respectively. The revenue is determined by the unit revenue of each spot cargo ($R^{SP}$) and the quantity transported ($q^{SP}$). The model minimizes the difference between cost and revenue terms. The objective function accounts for operations in both \textbf{P-1} and \textbf{P-2}: while parameters in \textbf{P-1} are deterministic, the expectations over all possible scenarios are considered in \textbf{P-2}.

Constraints \eqref{eq:route_select_1} and \eqref{eq:route_select_2} ensure that only ships assigned to route $r$ ($y_{vr}^1=1$) can make $x_{vre}>0$ trips on route $r$ with any speed $e$. Constraints \eqref{eq:route_select_3} to \eqref{eq:route_select_15} define the transition flow of ships in \textbf{P-1} and \textbf{P-2}. In detail, Constraint \eqref{eq:route_select_3} ensures that ships can only transfer from route $r$ to only one route $r'$ with one speed $e$ at the beginning of \textbf{P-1}. Constraint \eqref{eq:route_select_4} and \eqref{eq:route_select_7} together require if a ship decides to transfer to route $r'$ from any route $r$ with any speed $e$, it must sail on route $r'$ during \textbf{P-1}. Constraint \eqref{eq:route_select_5} and \eqref{eq:route_select_7} together ensure if a ship decide to sail on both route $r$ and $r'$ during \textbf{P-1}, it must perform a transfer from $r$ to $r'$ (or from $r'$ to $r$) with a selected speed $e$ during \textbf{P-1}. Constraint \eqref{eq:route_select_6} means the transfer matrix as depicted in Table \ref{tab:p1 route transition} must be symmetric. Here, an example to illustrate the purpose of Constraints \eqref{eq:route_select_4}, \eqref{eq:route_select_5}, \eqref{eq:route_select_6} and \eqref{eq:route_select_7} together is introduced. A ship $v$ is at route $r_0$ before planning, and it wants to transfer to route $r$ and $r'$ during \textbf{P-1}. According to Constraint \eqref{eq:route_select_4}, $y^1_{vr'}$ is 1 because $\sum_{e\in\mathcal{E}_v}y^{1,1}_{vrr'e}$ is 1. As it transfer from $r$ to $r'$, according to Constraints \eqref{eq:route_select_5} and \eqref{eq:route_select_6}, $\sum_{e\in\mathcal{E}_v}y^{1,2}_{vrr'e}$ and $\sum_{e\in\mathcal{E}_v}y^{1,2}_{vr're}$ are equal to 1, resulting in $y^1_{vr'}$ is 1. To lower its costs, the ship conducts all the transporting tasks on one route before it leaves for another one, which means it does not visit one route twice (otherwise the costs of transiting between routes increases with no benefits). In this case, the ship transferred twice during \textbf{P-1}, requiring the LHS of Constraint \eqref{eq:route_select_7} to be 2, and as the ship is on two routes during \textbf{P-1}, the RHS of Constraint \eqref{eq:route_select_7} must also be 2. Constraint \eqref{eq:route_select_7} here ensures that the ship does not visit route $r$ again after leaving it. Constraint \eqref{eq:route_select_8} and \eqref{eq:route_select_9} require that at the beginning of \textbf{P-2} (the end of \textbf{P-1}), a ship can be on route $r'$ at the beginning of \textbf{P-2} if and only if exactly one of the following two conditions is satisfied: it was transferred to $r'$ either from route $r$ at the beginning of \textbf{P-1} or from route $r''$ during \textbf{P-1}; if neither of these two conditions holds, or if both hold simultaneously, the ship must not be on route $r'$ at the beginning of \textbf{P-2}. Constraints \eqref{eq:route_select_10} to \eqref{eq:route_select_15} have the same meanings for \textbf{P-2} as Constraints \eqref{eq:route_select_3} to \eqref{eq:route_select_7} have for \textbf{P-1}. 

Constraints \eqref{eq:service-days_1} and \eqref{eq:service-days_2} ensure full utilization of ships including transporting cargoes, transiting between routes, ballast sailing, and staying at port. Constraints \eqref{eq:contract-frequency_1} and \eqref{eq:contract-frequency_2} guarantee the required frequency for every contract, while \eqref{eq:contract-demand_1} and \eqref{eq:contract-demand_2} ensure demand fulfillment. Constraints \eqref{eq:capacity_1} and \eqref{eq:capacity_2} respect the capacity limits of type $k$ on ship type $v$ on trade lane $i$, for both contractual and spot cargo. Constraints \eqref{eq:spot-cargo_1} and \eqref{eq:spot-cargo_2} restrict the spot cargo volume to the size of the spot market for each capacity type. Constraints \eqref{eq:demand_cii} (Demand-based CII) and \eqref{eq:supply_cii} (Supply-based CII) are alternative formulations and can be selectively implemented depending on the modeling needs, to ensure that the CII of all ships remains within the required boundary for each scenario $s$. Finally, constraints \eqref{eq:variable_1} to \eqref{eq:variable_5} define the domain of decision variables.

\section{Solution method}
\label{sec:algorithm}

In the model, \( R \) consists of trade lanes, and its size grows rapidly as the number of trade lanes increases. If a route can include at most \( C_{\max} \) trade lanes, the size of \( R \) becomes \( O\left(\sum_{l=1}^{C_{\max}} \binom{n}{l} (l-1)!\right) = O\left( \binom{n}{C_{\max}} (C_{\max}-1)! \right) \), where \( \binom{n}{l} \) is the number of ways to select \( l \) trade lanes from \( n \), and \( (l-1)! \) is the number of directed cycles formed by the selected \( l \) trade lanes. This rapid growth in \( R \) significantly increases the parameter space, complicating the solution process. For instance, the major oil transportation trade lanes are Middle East to Asia, Middle East to Europe, Middle East to North America, Africa to Europe, South America to North America, and Africa to Asia, totaling six in all. Therefore, these six trade lanes can have $\sum_{l=1}\binom{n}{l}(l-1)!=\sum_{l=1}^{6}\binom{6}{l}(l-1)!=415$ routes if $C_{max}=n$, which is a quite large parameter space.

However, minimizing the invalid sections and maximizing the valid sections of a route improves cost-effectiveness, as it reduces costs and carbon emissions on invalid sections, aiding both the objective function and CII compliance. Consequently, companies prefer routes with lower ballast ratios. To address computational complexity, we adopt a heuristic pre-processing and stepwise input method for \( R \).

First, the ballast ratio of each route is calculated using equation \eqref{form:ballst_ratio}, and the routes in $R$ are sorted in ascending order of ballast ratio. Routes are then added to the initial set if they include new trade lanes. For example, with three trade lanes (1, 2, 3) and routes \{1\}, \{2\}, \{3\}, \{1,2\}, \{1,3\}, \{2,3\}, \{1,2,3\}, sorting by ballast ratio yields the order \{2,3\}, \{3\}, \{2\}, \{1,2,3\}, \{1,2\}, \{1,3\}. The first route, \{2,3\}, is added to the initial set. Routes \{3\} and \{2\} are skipped as their trade lanes are already included. Route \{1,2,3\} is added as it includes a new trade lane (1). The final initial set consists of \{2,3\} and \{1,2,3\}.

{\small
\begin{align}
    \text{Ballast ratio of the route} = 1 - \frac{\text{Sum of lengths of all trade lanes in the route}}{\text{Length of the route including trade and ballast legs}} \label{form:ballst_ratio}
\end{align}
}

Next, input the initial set of routes into the solver and solve the problem. The remaining routes are sorted by ballast ratio and added iteratively to the route set, $m$ at a time, with the updated set re-entered into the solver. This process repeats until the solver's optimal solution stabilizes (change below a threshold) or a preset iteration limit is reached. The pseudocode for this algorithm is provided in Algorithm \ref{alg:algorithm}. To analyze the complexity of this algorithm, we denote the solution time of the MILP with $|R|$ routes as $T(|R|)$. Therefore, the time complexity of the algorithm can be expressed as $O\Big(|R| \log |R| + |R| \cdot N + \sum_{i=0}^K T(|R_o| + i \cdot step\_length)\Big)$,
where $|R|$ is the total number of routes, $N$ is the number of trade lanes, $K$ is the number of iterations limited by $max\_iteration\_times$, $T(s)$ is the solver's time complexity for input size $s$, and $|R_o|$ is the initial route set. Here, $O(|R| \log |R|)$ corresponds to the complexity of sorting, $O(|R| \cdot N)$ corresponds to the complexity of selecting the initial set of routes, and $O\Big(\sum_{i=0}^K T(|R_o| + i \cdot step\_length)\Big)$, corresponds to the complexity of solving the MILP iteratively with the selected route sets. Since inputting the full route set $R$ into the MILP at once is computationally expensive, the algorithm limits the number of iterations through $max\_iteration\_times$, effectively controlling computational costs. Numerical experiments show that this iterative approach is significantly faster than directly solving the full route set while achieving near-optimal solution.

\begin{algorithm}[H]
\label{alg:algorithm}
\SetAlgoLined
\caption{Heuristic Route Search Algorithm}
\KwIn{Set of routes $R$, \# of trade lanes $N$, threshold, step\_length, max\_iteration\_times}
\KwOut{Optimal decision variables and objective value}
\textbf{Step 1: Reorder the routes in set $R$ according to the ballast ratio}\;
    \For{$r \in R$}{
        Calculate the distance of all trade lane $i$ in $r$\;
        $L_r^{lane} = \sum_{i\in r} L_i$\;
        $ballast\_ratio_{r} \leftarrow 1 - \frac{L_r^{lane}}{L_r}$\;
        }
    Change the type of $R$ to list; $R_0 \leftarrow$ Sort($R$) by $ballast\_ratio$\;
\textbf{Step 2: Select the initial set of routes}\;
    Initial route set $R_1 \leftarrow \emptyset$\;
    \For{$r \in R_0$}{
        $I \leftarrow \text{All trade lane $i$ which can be served by routes in $R_1$}$ \;
        $I' \leftarrow \text{All trade lane $i$ which can be served by route $r$}$\;
        \If{$card(I)=N$}{break\;}
        \If{$I' \setminus I \neq \emptyset$}{
            $R_1 \leftarrow R_1\  \text{append $r$}$, $R_0 \leftarrow R_0\ \text{delete $r$}$\;
        }
    }
\textbf{Step 3: Enter $R_1$ into the model and perform iterations.
}\;
    $Decision\_variable\_values_0$, $objective\_value_0$ $\leftarrow$ Model($R_1$)\;
    $iteration \leftarrow 0$\;
    \For{each batch $R_s \subseteq R_0\setminus R_1$ ($card(R_s)=step\_length$)}{
        $iteration \leftarrow iteration + 1$\;
        $Decision\_variable\_values$, $objective\_value$ $\leftarrow$ Model($R_1 \cup R_s$)\;
        \If{$\| objective\_value - objective\_value_0 \| < threshold \times \| objective\_value_0 \|$ or $iteration \geq max\_iteration\_times$}{\textbf{return} $Decision\_variable\_values$, $objective\_value$\;}
        $Decision\_variable\_values_0$, $objective\_value_0$ $\leftarrow$ $Decision\_variable\_values$, $objective\_value$\;
        $R_1$ $\leftarrow$ $R_1 \cup R_s$
    }
\end{algorithm}

The proposed model and algorithm are implemented in Python using the Gurobi Optimizer 12.0.1 via the gurobipy interface and solved on a high-performance computing server (FusionServer G5500 V6) equipped with 128 CPU cores and 2 TB of RAM.

\section{Numerical analysis}
\label{sec:case_study}

This section analyzes the problem by evaluating outcomes under two CII policies, and the impact of future uncertainties using a generated instance.

\subsection{Instance generation}
\label{sec:instance}

Due to the lack of actual data, we approximate our dataset based on previous research such as \cite{wang2018planning}. We assume 15 ships, 5 trade lanes, 3 speed choices, 3 contracts, 2 capacity types, and 13 scenarios in \textbf{P-2}, as shown in Table \ref{tab:relationship}. Decisions are made at the end of April, with \textbf{P-1} spanning May to August and \textbf{P-2} from September to December.

We first generate basic parameters, including vessel availability, route lengths, speed options, contract requirements, spot market cargoes, and capacity volumes. Deterministic values for \textbf{P-1} are derived from these, e.g., $T_{vre}$ is approximated as $T_{vre} = \frac{L_r}{e}$ with added layover time. Parameter ranges are detailed in Tables \ref{table:parameters_1} and \ref{table:parameters_2}. For the CII standard, we adopt the example from \cite{IMO_MEPC76_15_Add2_2021} for a bulk ship. To compare the outcomes under different CII stringency levels, we set the lenient standard at 11.8 g/(ton·nmile), corresponding to rate D, and the stricter standard at 8.6 g/(ton·nmile), corresponding to rate A.

In \textbf{P-2}, uncertainties include fuel costs, market demand, and freight rates, linked to parameters such as $C^{RT}_{s}$ (fuel costs), $D_{cs}$ and $D^{SP}_s$ (market demand), and $R^{SP}_s$ (spot transportation revenue). These factors are interrelated, as shown in Table \ref{tab:relationship} ($\uparrow$: rising, $\downarrow$: falling, \textemdash: unchanged). Higher trade demand typically increases freight rates \citep{bai2021tanker,lun2013demand}, while rising fuel prices, a key cost component, also raise freight rates as costs are passed to customers \citep{xu2011dynamics}. However, when fuel prices and demand move in opposite directions, freight rates may vary in any direction.

\begin{table}[!ht]
    \centering
    \begin{tabular}{ccc}
        \toprule
        Fuel price & Market demand & Freight rate \\ \hline
        \textemdash & \textemdash & \textemdash \\
        \textemdash & $\uparrow$ & $\uparrow$ \\ 
        \textemdash & $\downarrow$ & $\downarrow$ \\ 
        $\uparrow$ & \textemdash & $\uparrow$ \\ 
        $\uparrow$ & $\uparrow$ & $\uparrow$ \\ 
        $\uparrow$ & $\downarrow$ & Any direction \\ 
        $\downarrow$ & \textemdash & $\downarrow$ \\ 
        $\downarrow$ & $\uparrow$ & Any direction \\ 
        $\downarrow$ & $\downarrow$ & $\downarrow$ \\ 
        \bottomrule
    \end{tabular}
    \caption{The interrelationship of changes among fuel price, market demand, and freight rate}
    \label{tab:relationship}
\end{table}

\begin{table}[!ht]
    \centering
    \begin{tabular}{cccc}
        \toprule
        Scenario & Fuel price & Market demand & Freight rate \\ \hline
        1 & \textemdash & \textemdash & \textemdash \\
        2 & \textemdash & $\uparrow$ & $\uparrow$ \\ 
        3 & \textemdash & $\downarrow$ & $\downarrow$ \\ 
        4 & $\uparrow$ & \textemdash & $\uparrow$ \\ 
        5 & $\uparrow$ & $\uparrow$ & $\uparrow$ \\ 
        6 & $\uparrow$ & $\downarrow$ & \textemdash \\
        7 & $\uparrow$ & $\downarrow$ & $\uparrow$ \\
        8 & $\uparrow$ & $\downarrow$ & $\downarrow$ \\ 
        9 & $\downarrow$ & \textemdash & $\downarrow$ \\ 
        10 & $\downarrow$ & $\uparrow$ & \textemdash \\
        11 & $\downarrow$ & $\uparrow$ & $\uparrow$ \\
        12 & $\downarrow$ & $\uparrow$ & $\downarrow$ \\ 
        13 & $\downarrow$ & $\downarrow$ & $\downarrow$ \\
        \bottomrule
    \end{tabular}
    \caption{Generated scenarios}
    \label{tab:scenarios_0}
\end{table}

We define each scenario qualitatively (Table~\ref{tab:scenarios_0}) and then quantify changes as follows. To ease computation while maintaining representativeness, we do not generate variables from full probability distributions, which would require a large number of scenarios to adequately capture the variability of each parameter. Instead, we approximate these distributions by selecting a representative scale factor for each direction of change. For example, parameters expected to remain unchanged are kept at a selected value in \textbf{P-2}, those expected to increase are scaled to 120\% of the value, those expected to decrease are scaled to 80\% of the value. This approach yields a representative set of future scenarios that balances realism and computational tractability. A more detailed explanation and sensitivity analysis regarding the number of scenarios is provided in Appendix~\ref{appendix:scenarios}. In practice, shipping companies can first identify key uncertainty sources (e.g., fuel price, market demand, freight rate) and discretize them into states (e.g., rising, stable, falling) based on historical data and expert judgment. After establishing economically plausible interrelationships (Table~\ref{tab:relationship}) to filter out implausible combinations, quantitative values are assigned to each state and probabilities are assigned to each scenario, using historical distributions or managerial forecasts. The number of scenarios can be chosen by the shipping company depending on the richness of their historical data. In our example, this is simplified by applying scaling factors of 120\%, 100\%, and 80\% to the rising, stable, and falling states in \textbf{P-1} to generate the corresponding values in \textbf{P-2}, with each scenario assigned an equal probability.

\begin{table}[htbp]
\centering
\caption{The value/range of the deterministic parameters in the case.}
\label{table:parameters_1}
\begin{tabular}{@{}lp{10cm}p{4.3cm}}
\toprule
Parameters & & Value/range \\ \midrule
$M_{v}^{1}$ & No. of available days for ship $v$ during \textbf{P-1}. & 120[days] \\
$M_{v}^{2}$ & No. of available days for a ship of type $v$ during \textbf{P-2}. & 120[days] \\
$Q_{vk}$ & Volume of capacity type k installed on ship $v$. & 10,331-49,944[t]\\
$T_{vre}$ & Total travel time for ship type $v$ to complete a round route $r$ with speed alternative $e$, including sailing time and time spent at ports, etc. & 32-175[days] \\
$F_c^{1}, F_c^{2}$& Frequency requirement of contract$ c$ in \textbf{P-1}. Frequency requirement of contract $c$ in \textbf{P-2}. & 3 or 6 \\
$D_c$ & Demand of contract $c$ in \textbf{P-1}. & 144.87-176.15K[t] \\
$D_{ik}^{SP}$ & Volume of spot cargo available on trade lane $i$ compatible with capacity type $k$ in \textbf{P-1}. & 15.22-24.62K[t] \\
$C_{vre}^{RT}$ & Cost for ship $v$ to complete a round route $r$ with speed alternative $e$ in \textbf{P-1}, including fuel cost, port fees, canal tolls, etc. & \$0.21-3.31M \\
$C_{vr_0re}^{TF}$ & Cost for ship $v$ to transfer route from $r_0$ to $r$ with speed alternative $e$ in \textbf{P-1}, including fuel cost, port fees, canal tolls, etc. & \$0.24-4.88M \\
$C_v^{ballast}$ & Cost for ship $v$ to ballast sail per day during \textbf{P-1}. & \$8.48-27.82K \\
$C_v^{port}$ & Cost for ship $v$ at port per day during \textbf{P-1}. & \$0.85-2.78K\\
$R_{ik}^{SP}$ & Revenue of delivering one ton of spot cargo with capacity type $k$ on trade lane $i$ in \textbf{P-1}. & \$49-120 \\
$E_{vre}$ & Carbon emission of ship $v$ to complete a round trip on route $r$ with speed alternative $e$. & 1.64-25.92B[g] \\ 
$E_{vr_0re}$ & Carbon emission of ship $v$ to complete a transfer trip from route $r_0$ to route $r$ with speed $e$. & 1.86-38.14B[g] \\
$E_v^{ballast}$ & Carbon emission of ship $v$ ballast sailing per day. & 36.75-122.12M[g] \\
$E_v^{port}$ & Carbon emission of ship $v$ at port per day. & 3.67-12.21M[g]\\
$E_v^0$ & Carbon emission of ship $v$ before planning. & 8.25-28.64B[g]\\
$CII_v$ & The Carbon Intensity Indicator standard for ship $v$. & lenient: 11.8[g/(t$\cdot$nmile)]; stricter: 8.6[g/(t$\cdot$nmile)]\\
$W_v^0$ & Working intensity (laden sailing distance $\times$ corresponding loads) of ship $v$ before planning. & 644.78-2,086.43M[t$\cdot$ nmile]\\
$L_i$ & Laden sailing distance of trade lane $i$. & 6,144-9,904[nmile] \\
$L_v^{ballast}$ & Distance of ship $v$ per day when ballast sailing (the distance of all ships are the same as they prefer ballast sailing in the lowest speed). & 288[nmile] \\
$L_r^{route}$ & Distance of route $r$. & 12.29-49.64K[nmile] \\
$L_{rr'}$ & Distance of switching from route $r$ to $r'$. & 1,004-5,990[nmile] \\
$y_{vr_0}^{1,0}$ &0-1 variable, whether ship $v$ on route $r_0$ before planning. & 0 or 1 \\
\bottomrule
\end{tabular}
\begin{tablenotes}
    \item Note: \$ for USD, nmile for nautical miles, t for tons, kg for kilograms, g for grams, B for billions, M for millions, and K for thousands.
 \end{tablenotes}
\end{table}

\begin{table}[htbp]
\centering
\caption{The value/range of the stochastic parameters in the case.}
\label{table:parameters_2}
\begin{tabular}{@{}lp{11.5cm}l@{}}
\toprule
Parameters & & Value/range \\ \midrule
$p_s$ & The probability of scenarios $s$ taking place during \textbf{P-2}. & $\frac{1}{13}$ \\
$D_{cs}$ & Stochastic demand of contract $c$ during \textbf{P-2}. & 115.90-211.38K[t] \\
$D_{iks}^{SP}$ & Volume of spot cargo available on trade lane $i$ compatible with capacity type $k$ during \textbf{P-2} in scenario $s$. & 12.18-29.54K[t] \\
$C_{vr1es}^{RT}$ & Cost for ship $v$ to complete a round trip on route $r_1$ with speed alternative $e$ during \textbf{P-2} in scenario $s$. & \$0.17-3.98M \\
$C_{vrr_1es}^{TF}$ & Cost for ship $v$ to transfer route from $r$ to $r_1$ with speed alternative $e$ during \textbf{P-2} in scenario $s$. & \$0.19-5.85K \\
$C_{vs}^{ballast}$ & Cost for ship $v$ to ballast sail per day during \textbf{P-2} in scenario $s$. & \$6.79-33.38K \\
$C_{vs}^{port}$ & Cost for ship $v$ at port per day during \textbf{P-2} in scenario $s$. & \$0.68-3.34K \\
$R_{iks}^{SP}$ & Revenue of delivering one ton of spot cargo with capacity type $k$ on trade lane $i$ during \textbf{P-2} in scenario $s$. & \$39-144 \\
\bottomrule
\end{tabular}
\begin{tablenotes}
    \item Note: \$ for USD, t for tons, kg for kilograms, M for millions, and K for thousands.
 \end{tablenotes}
\end{table}

\subsection{Algorithm performance}

To evaluate the effectiveness of our Algorithm~\ref{alg:algorithm} in Section~\ref{sec:algorithm}, we generated multiple parameter sets for trade lane quantities $N$ of 2, 3, 4, and 5. Effectiveness is assessed by comparing average objective function values and running times between our algorithm and exact solutions from commercial solvers. The $threshold$ in Algorithm \ref{alg:algorithm} is set at 5\%, $max\_iteration\_times$ at 3, and $step\_length$ at $\frac{3}{2}N$ with 10 parameter sets generated and tested per trade lane quantity. The parameters are chosen through experiments on small instances to ensure the quality of solutions and efficiency. To simplify computation and expedite results, we reduced the number of ships to 8 and contracts to 2. As shown in Figure \ref{fig:alg_2_5}, the running time difference increases with the number of trade lanes, while the error in objective function values remains small. The algorithm performs better with larger parameter spaces (e.g., four or five trade lanes) than with smaller ones.

\begin{figure}[!h]
    \begin{subfigure}[b]{0.5\textwidth}
        \centering
        \includegraphics[width=\linewidth]{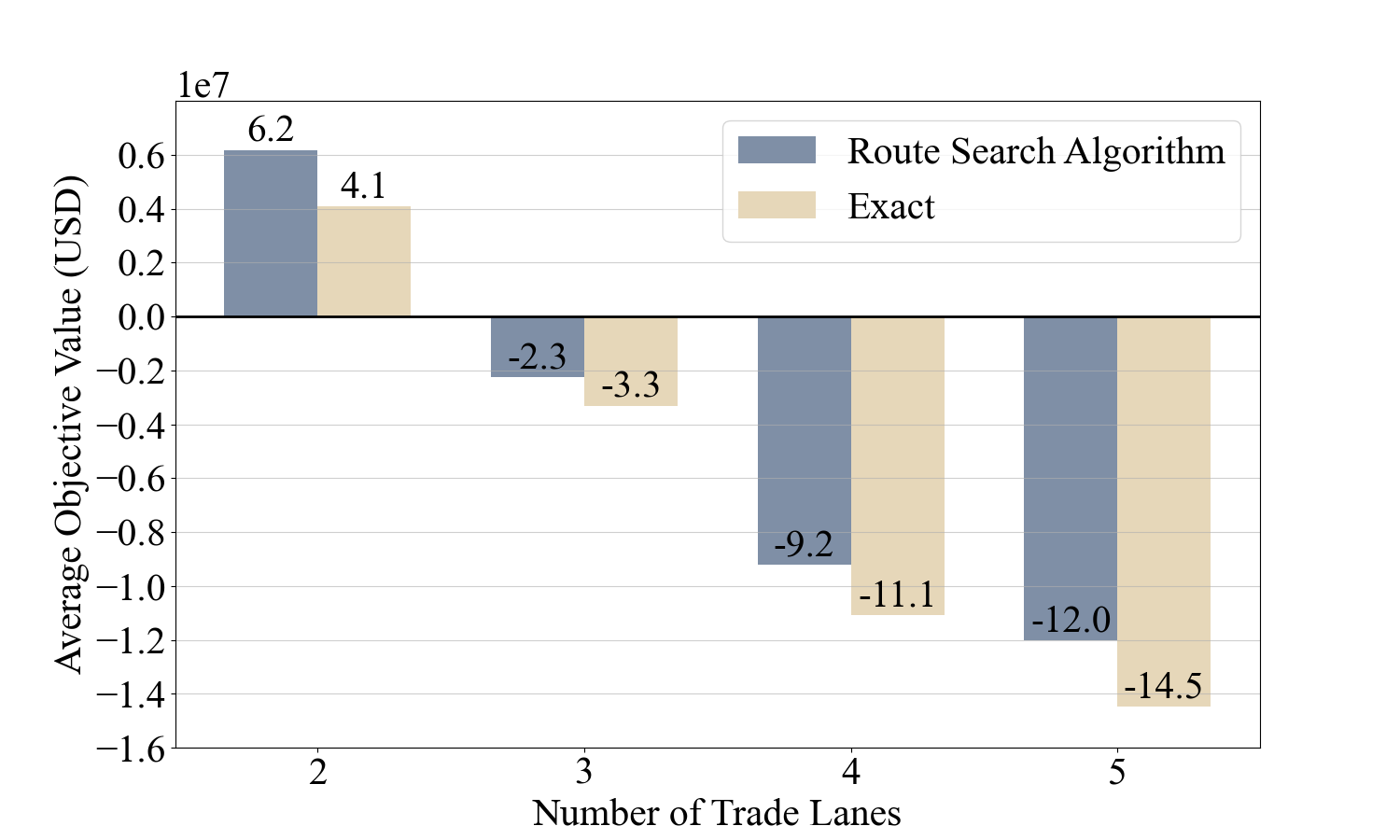}
        \caption{Average objective values ($N$=2, 3, 4 and 5)}
        \label{fig:alg_2_5_obj}
    \end{subfigure}%
    \hfill % Add some space between subfigures
    \begin{subfigure}[b]{0.5\textwidth}
        \centering
        \includegraphics[width=\linewidth]{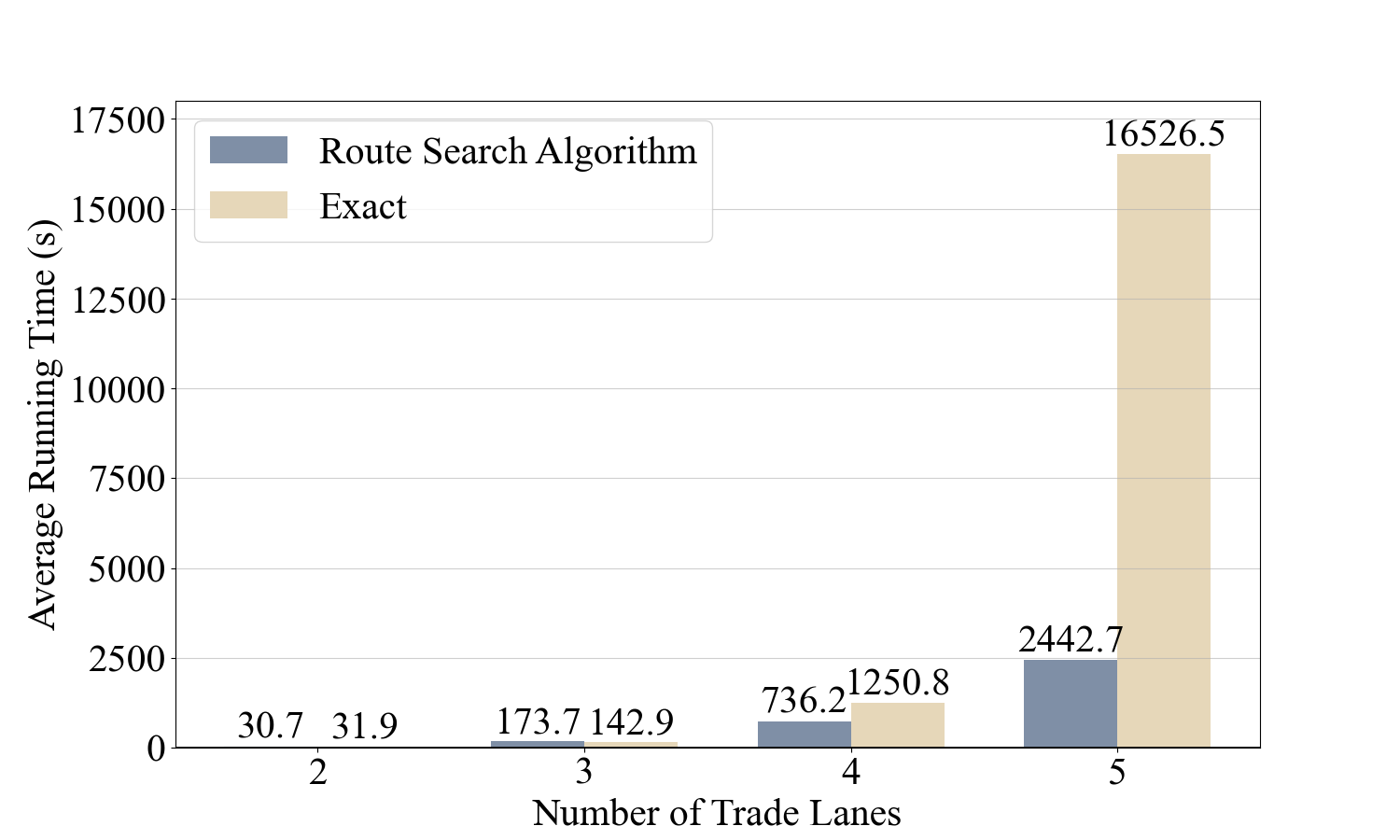}
        \caption{Average running time ($N$=2, 3, 4 and 5)}
        \label{fig:alg_2_5_time}
    \end{subfigure}
    \caption{Difference of (a) average objective values and (b) running time of two solution methods under different trade lane quantities.}
    \label{fig:alg_2_5}
\end{figure}

It is obvious that the above result shows that our algorithm outperforms the exact approach using Gurobi Optimizer in terms of computational time. To demonstrate that the objective values obtained from our heuristic algorithm are sufficiently close to those from the exact method, we note that the original objective function is defined as $(\text{cost} - \text{revenue})$. Since this value is relatively small compared to the magnitude of the cost and revenue terms themselves, directly comparing absolute differences in the objective function can lead to numerically unstable or misleading conclusions, where even small deviations result in disproportionately large relative errors.  For example, consider a case where the exact solution yields a cost of $\$1.01 \times 10^8$ and a revenue of $\$1.00 \times 10^8$, resulting in an objective value of $\$1 \times 10^6$. If the heuristic solution yields a slightly lower cost of $\$1.00 \times 10^8$ while the revenue remains unchanged, the resulting objective becomes 0. Although the heuristic only introduces a 1\% error in the cost term, a direct comparison of the objective values would suggest a 100\% error:
$\frac{|\text{Objective}^{(heuristic)} - \text{Objective}^{(exact)}|}{|\text{Objective}^{(exact)}|} = \frac{|0 - 1 \times 10^6|}{1 \times 10^6} \times 100\% = 100\%$. This clearly overstates the actual deviation introduced by the heuristic and motivates the need for a more robust, normalized comparison metric. To address this, we introduce a normalized deviation metric that evaluates the closeness between the heuristic and exact solutions in a scale-invariant manner: $$\text{Normalized Deviation} = \frac{|\text{cost}^{(heuristic)} - \text{cost}^{(exact)}| + |\text{revenue}^{(heuristic)} - \text{revenue}^{(exact)}|}{\text{cost}^{(exact)} + \text{revenue}^{(exact)}}.$$ 
We visualize the distribution of normalized deviations across different trade lane scales using boxplots, where each box summarizes the deviations over 10 test instances as mentioned above for the corresponding scale in Figure~\ref{fig:boxplot_deviation}. As shown, in all cases the deviations are consistently concentrated around 5\%. This suggests that the deviations produced by the heuristic algorithm are not significantly greater than 5\%, confirming that the solution quality is statistically indistinguishable from the exact method within a 5\% margin.

\begin{figure}[!h]
    \centering
    \includegraphics[width=0.5\linewidth]{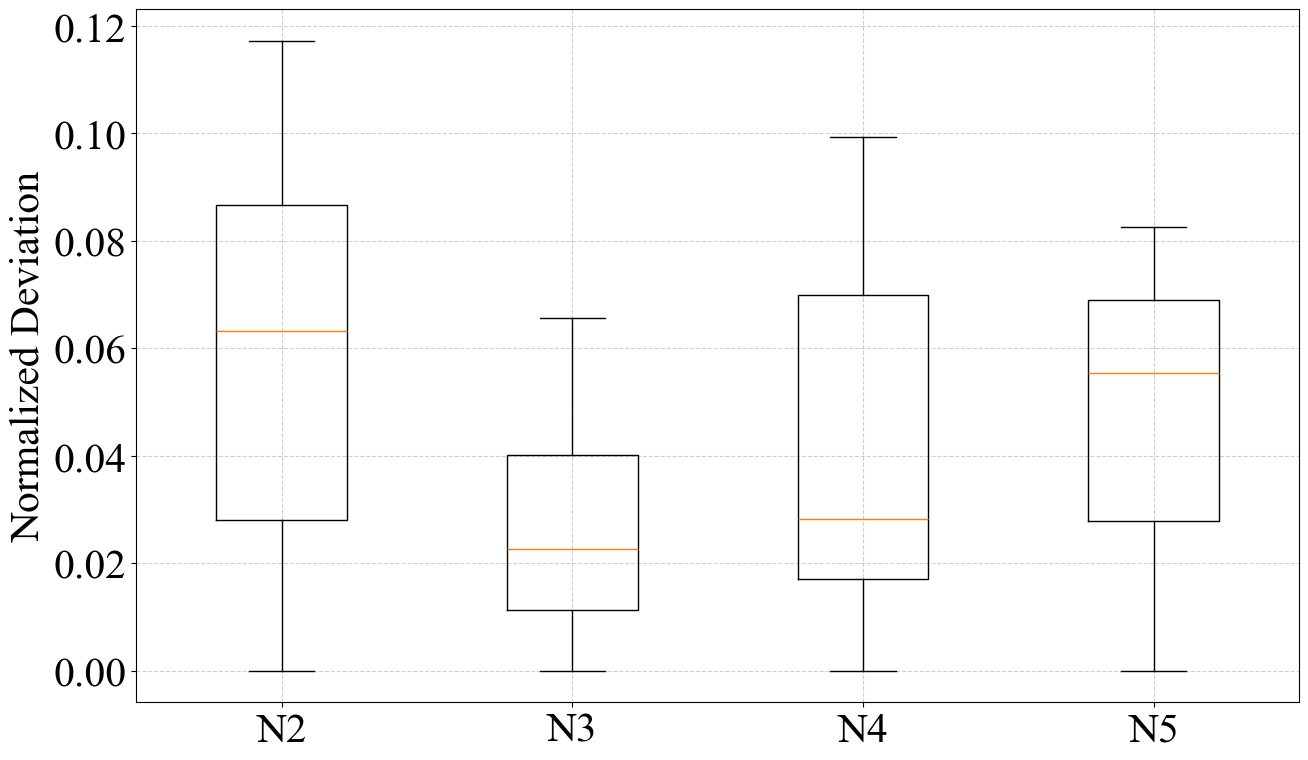}
    \caption{Normalized deviation between heuristic route search algorithm and exact solver ($N$=2, 3, 4 and 5)}
    \label{fig:boxplot_deviation}
\end{figure}

To formally assess whether the deviations are statistically below an acceptable threshold, we conduct one-sided Wilcoxon signed-rank tests for each trade lane scale, with 10 results obtained from 10 datasets for each scale. The hypotheses are defined as follows:
$$H_0: \text{Median deviation} \leq 0.05, \quad \quad H_1: \text{Median deviation} > 0.05.$$
The results are summarized in Table~\ref{tab:wilcoxon_test}. Across different trade lane scales, the $p$-values exceed 0.05, and thus we fail to reject the null hypothesis. This suggests that the deviations from the heuristic algorithm are not significantly greater than 5\%, confirming that the solution quality is statistically indistinguishable from the exact method within a 5\% margin.

\begin{table}[htbp]
    \centering
    \begin{tabular}{lcc}
        \toprule
        $N$ & Test Statistic & $p$-value \\
        \midrule
        2 & 30 & 0.4229 \\
        3 & 5  & 0.9932 \\
        4 & 22 & 0.7217 \\
        5 & 28 & 0.5000 \\
        \bottomrule
    \end{tabular}
    \caption{One-sided Wilcoxon signed-rank test results for normalized deviation against 5\% threshold}
    \label{tab:wilcoxon_test}
\end{table}

To demonstrate the underlying principle behind the high efficiency of our algorithm, we select a parameter set with five trade lanes and compare the objective values and ship routing decisions between the two methods. This comparison highlights the effectiveness of our algorithm in solving large-scale problems. The selected data include 6 ships, 5 trade lanes formulating 41 feasible routes, 3 speed choices, 3 contracts, 2 capacity types, and 13 scenarios. As described in Section \ref{sec:algorithm}, routes with lower ballast ratios typically yield better results. Using the input method, we filter the initial route set and reorder remaining routes by ballast ratio, appending them sequentially. Figure \ref{fig:alg_5} shows the variation in objective values and running time as the number of routes increases. The algorithm stops at 25 routes, where the relative change in objective values from 20 to 25 routes is under 1\%.

\begin{figure}[!h]
    \begin{subfigure}[b]{0.47\textwidth}
        \centering
        \includegraphics[width=\linewidth]{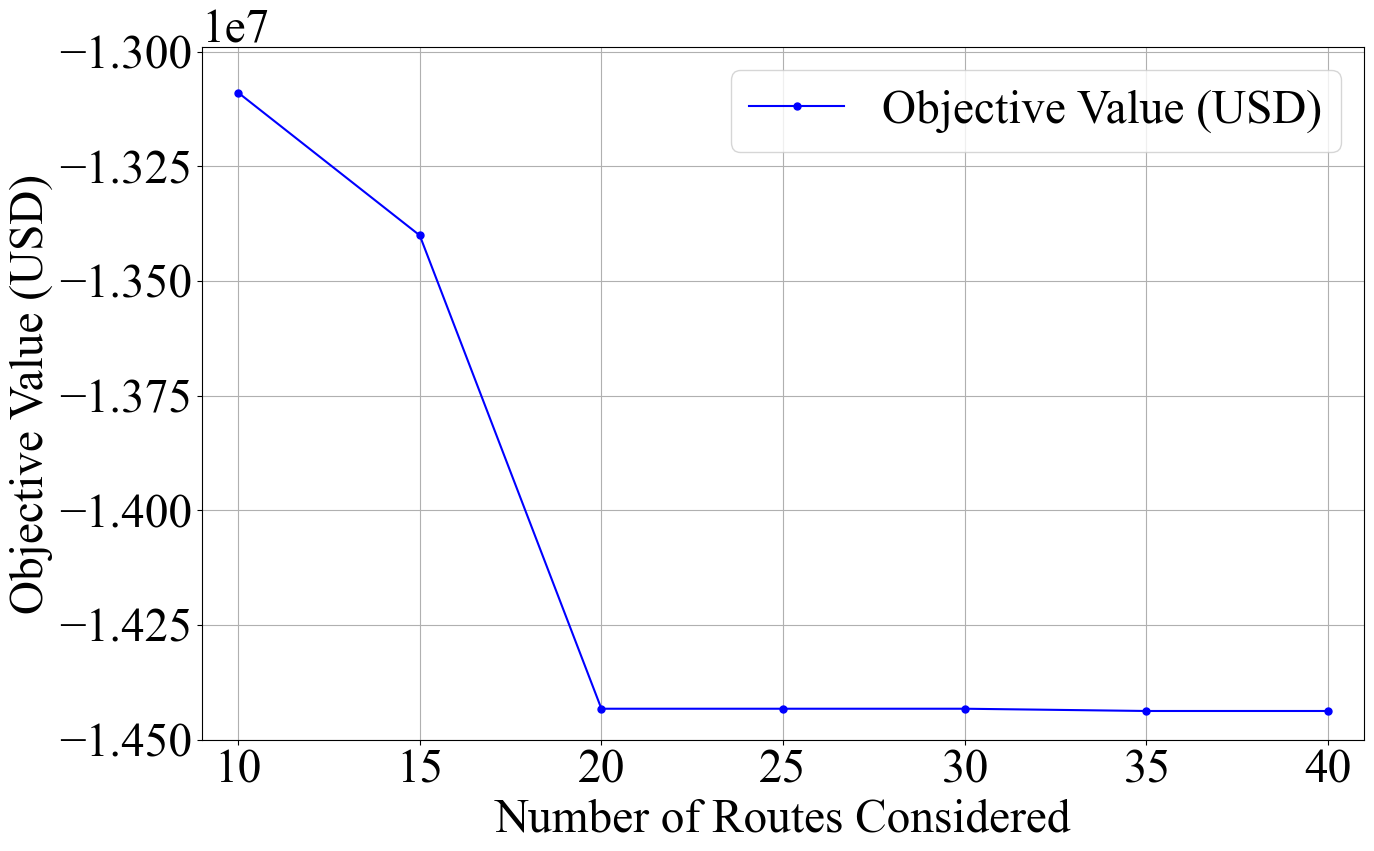}
        \caption{Objective values ($N$=5)}
        \label{fig:alg_5_obj}
    \end{subfigure}%
    \hfill % Add some space between subfigures
    \begin{subfigure}[b]{0.53\textwidth}
        \centering
        \includegraphics[width=\linewidth]{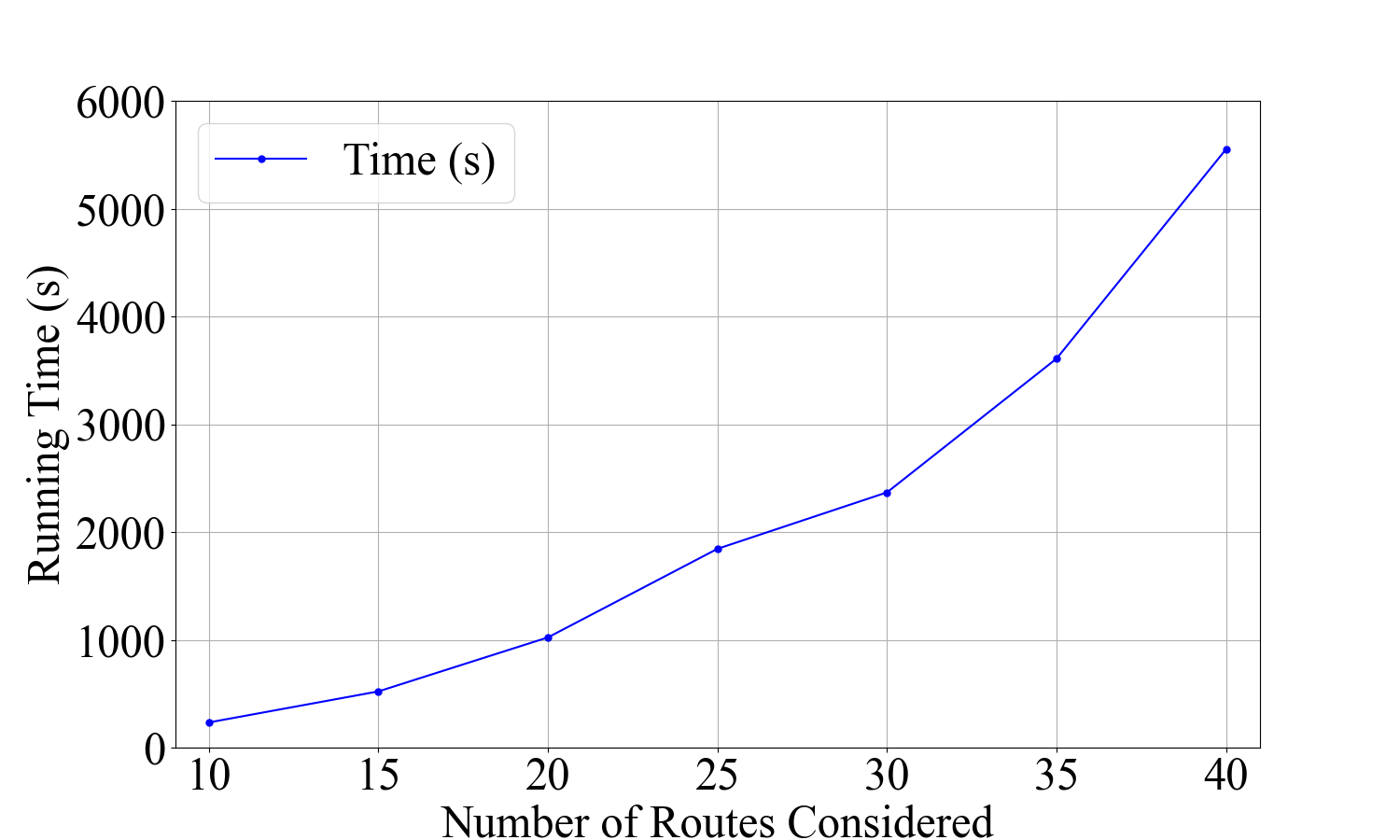}
        \caption{Running time ($N$=5)}
        \label{fig:alg_5_time}
    \end{subfigure}
    \caption{Changes of (a) objective values and (b) running time with quantities of routes considered (5 trade lanes).}
    \label{fig:alg_5}
\end{figure}

The minimal change in the objective function as the number of routes increases has specific causes. Since the algorithm sequentially incorporates routes based on their ballast ratios, ships prioritize earlier-listed routes. In the experiment shown in Figure \ref{fig:alg_5}, the route labels (deployment decisions) selected by all ships in \textbf{P-1} are displayed in Figure \ref{fig:alg_decision}. While adding more routes expands the route space, ships still prefer earlier routes, limiting the impact on solution quality. This experiment numerically verifies the algorithm's rationale and logic.

\begin{figure}[!h]
    \centering
    \includegraphics[width=0.75\linewidth]{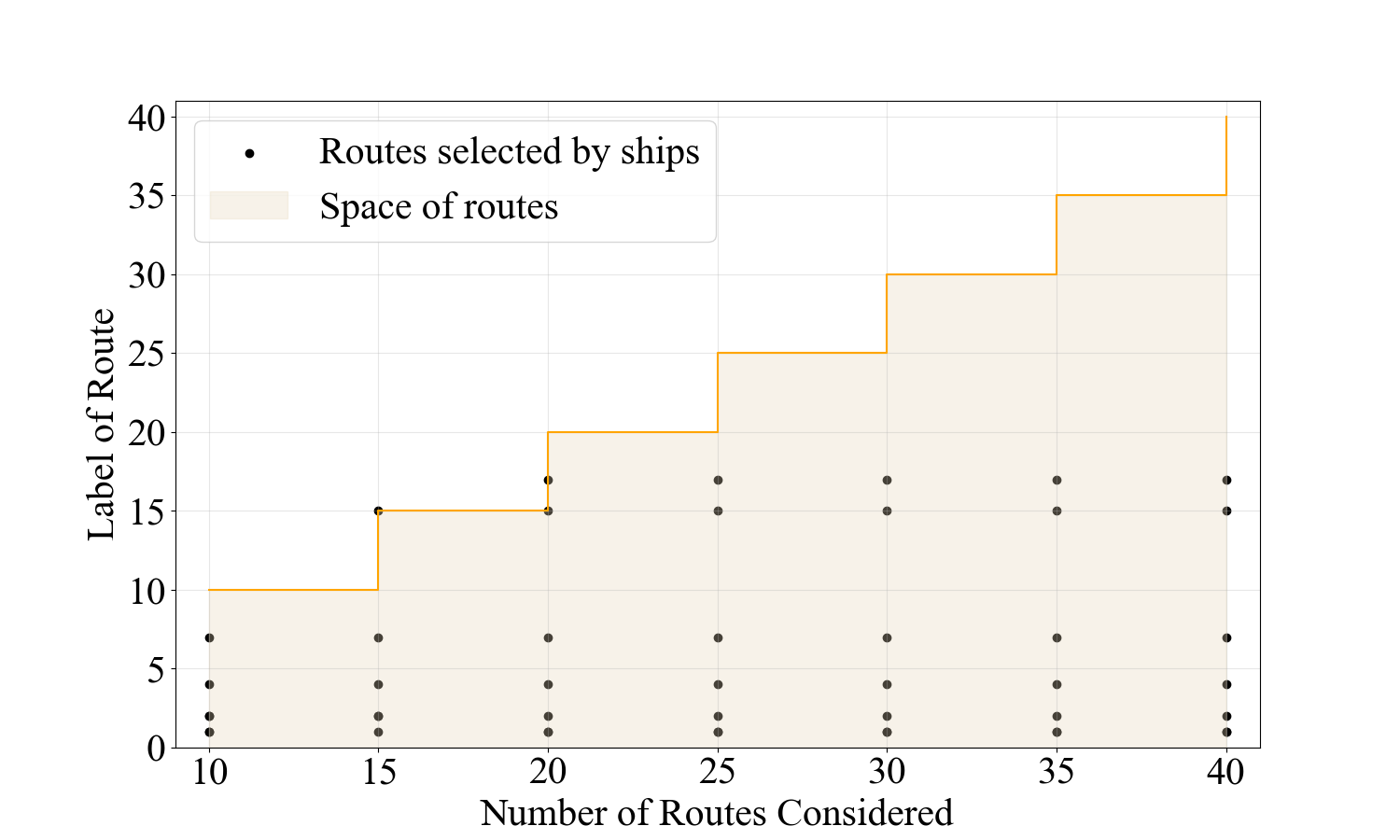}
    \caption{Labels of routes selected by ships during \textbf{P-1} under different quantities of routes considered into the model (5 trade lanes).}
    \label{fig:alg_decision}
\end{figure}

\subsection{Results under two CII forms}
\label{sec:compare_cii}

In Sections \ref{sec:compare_cii} and \ref{sec:ana_scenario}, we use a same randomly generated dataset to conduct case studies. The data include 15 ships, 5 trade lanes formulating 41 feasible routes, 3 speed choices, 3 contracts, 2 capacity types, and 13 scenarios. Here, we compare results under Demand-based CII and Supply-based CII constraints.

In this section, we examine two CII standards: one allowing feasible solutions under both forms, and another causing infeasibility under Demand-based CII but feasible under Supply-based CII by adjusting ballast segments outside transportation tasks, which leads to the Supply-based CII paradox. We analyze results for three scenarios: lenient Demand-based and Supply-based CII, and stricter Supply-based CII.

First, five tramp operation indicators are introduced in Table \ref{tab:tramp_indicators}. Then, CargoQuant, BallastRatio, AvgSpeed, and CO2Emission values during \textbf{P-1} are compared across the three scenarios in Figure \ref{fig:comparison_indicators}.

\begin{table}[!h]
    \centering
    \begin{tabular}{cccc}
    \toprule
    \textbf{Indicator Type} & \textbf{Indicator} & \textbf{Meaning} & \textbf{Unit} \\ \toprule
    Cargo Selection & CargoQuant & Total quantities of served cargoes in spot market & ton \\ \hline
    Idle Option & BallastRatio & The proportion of idle time spent on ballast sailing & \% \\ \hline
    Speed Optimization & AvgSpeed & Average sailing speed in the year & nmile/h \\ \hline
    \multirow{2}{*}{Ultimate Objectives} & CO2Emission & Total CO2 emissions & kg \\ \cline{2-4}
     & TotalProfit & Total profits & USD \\ 
    \bottomrule
    \end{tabular}
    \caption{Indicators of tramp operations under}
    \label{tab:tramp_indicators}
\end{table}

\begin{figure}[!h]
    \begin{subfigure}[b]{0.5\textwidth}
        \centering
        \includegraphics[width=\linewidth]{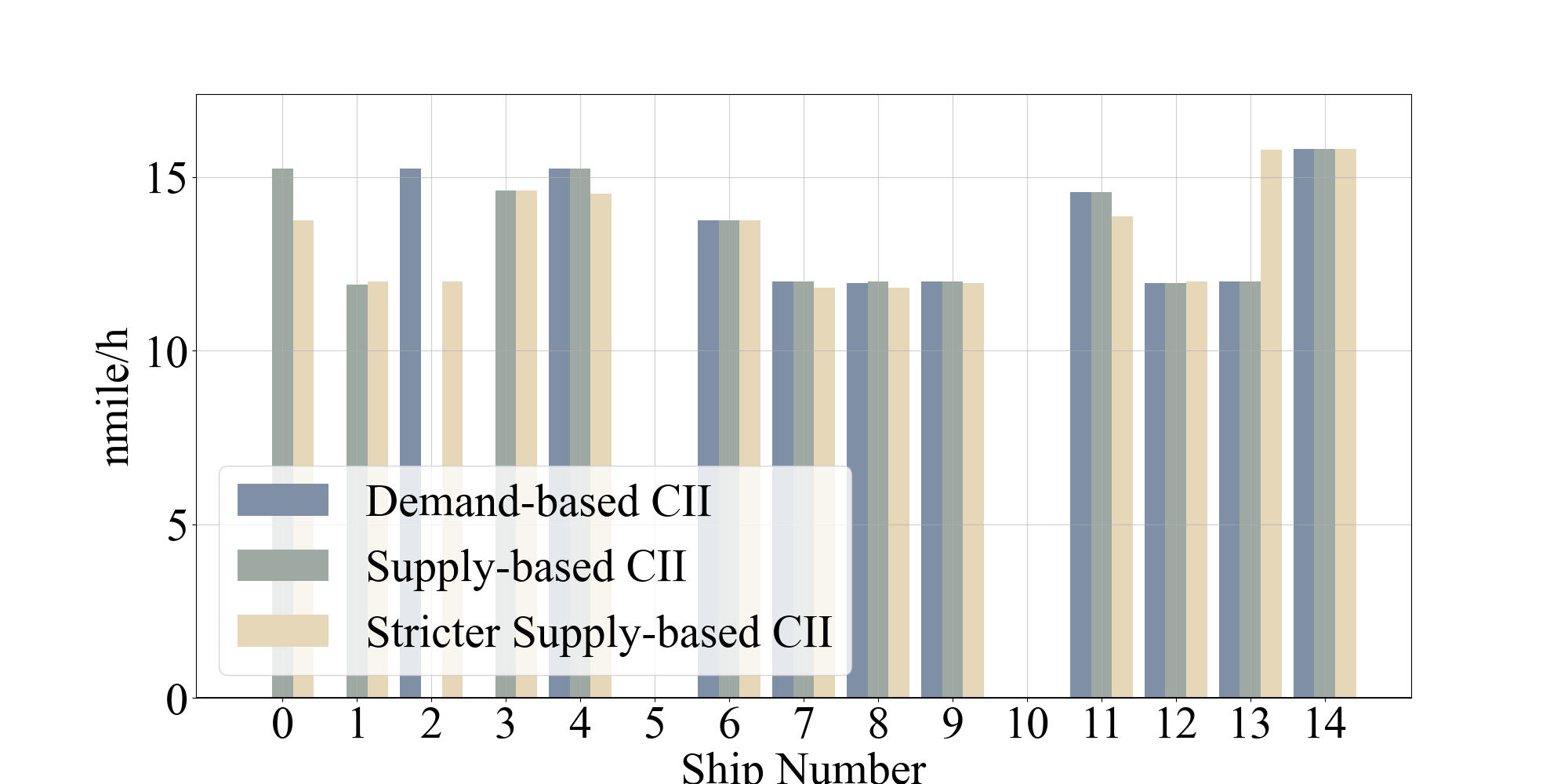}
        \caption{AvgSpeed in \textbf{P-1} vs. Ship number}
        \label{fig:compare_speed}
    \end{subfigure}%
    \hfill
    \begin{subfigure}[b]{0.5\textwidth}
        \centering
        \includegraphics[width=\linewidth]{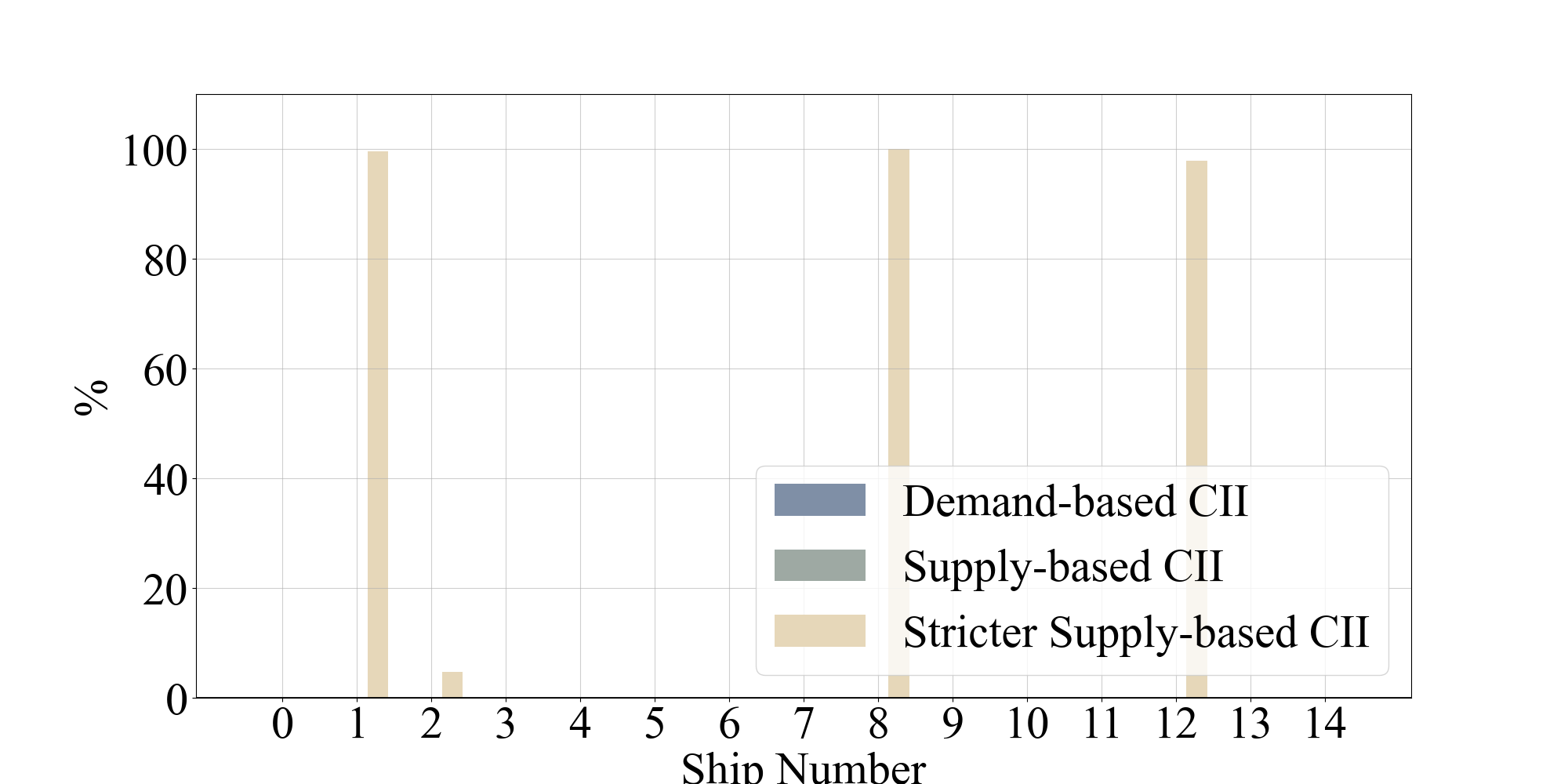}
        \caption{BallastRatio in \textbf{P-1} vs. Ship number}
        \label{fig:compare_ballast}
    \end{subfigure}
    \hfill
    \begin{subfigure}[b]{0.5\textwidth}
        \centering
        \includegraphics[width=\linewidth]{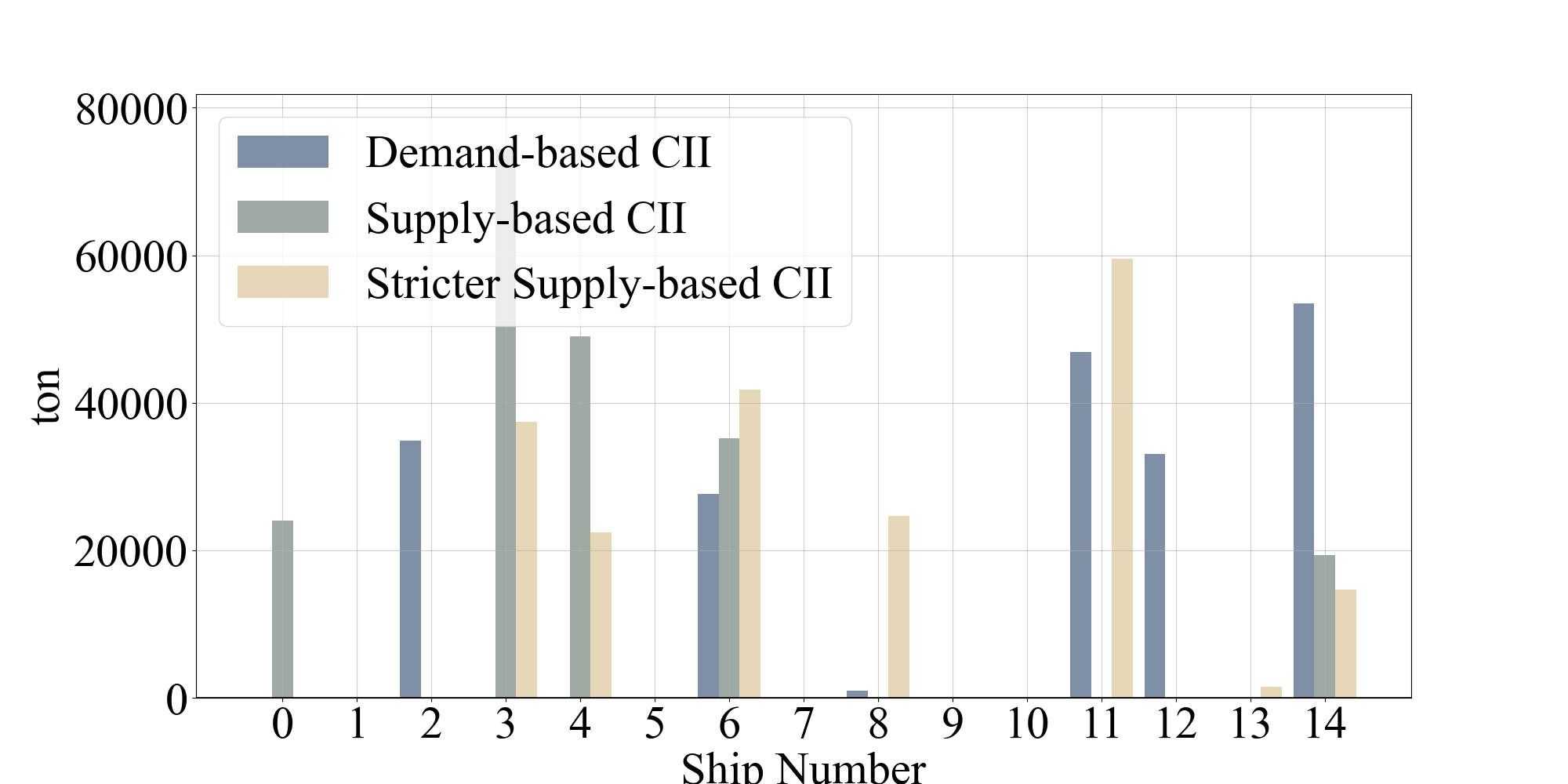}
        \caption{CargoQuant in \textbf{P-1} vs. Ship number}
        \label{fig:compare_load}
    \end{subfigure}
    \hfill
    \begin{subfigure}[b]{0.5\textwidth}
        \centering
        \includegraphics[width=\linewidth]{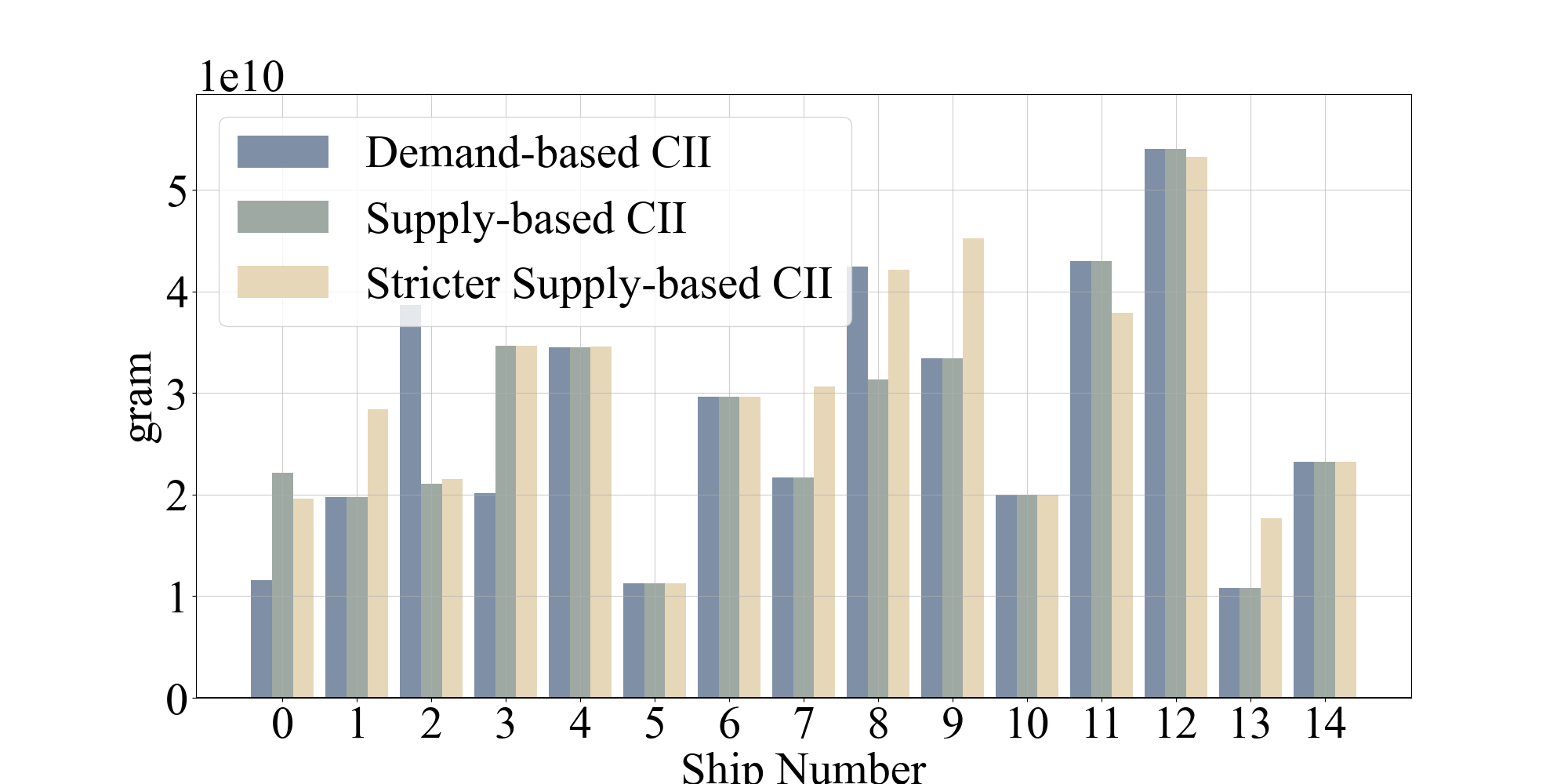}
        \caption{CO2Emission in \textbf{P-1} vs. Ship number}
        \label{fig:compare_carbon}
    \end{subfigure}
    \caption{Comparison of Indicators under Three CII Standards}
    \label{fig:comparison_indicators}
    \caption*{\raggedright \setstretch{1}Note: \textbf{P-1} denotes the first planning stage. 
    \textit{AvgSpeed} represents the average sailing speed during the year; 
    \textit{BallastRatio} is the proportion of idle time spent on ballast sailing; 
    \textit{CargoQuant} indicates the total quantity of cargoes served in the spot market; 
    and \textit{CO2Emission} denotes total carbon dioxide emissions.}
\end{figure}

As shown in Figure~\ref{fig:compare_speed}, ship speeds under the lenient Demand-based and Supply-based CII are almost identical, except for a few ships that remain at port with zero speed. However, under the stricter Supply-based CII, several ships (e.g., Ship~0, 2, 4, and~11) deliberately reduce their sailing speeds to lower fuel consumption and emissions. In contrast, under the stricter Supply-based CII, some of them (e.g., Ship~1, 2, 8, and~12) undertake additional ballast sailings to satisfy the tightened carbon intensity requirements, as illustrated in Figure~\ref{fig:compare_ballast}.

Figure~\ref{fig:compare_load} illustrates the cargo volume transported by each vessel under different CII regulatory schemes. While the total cargo volume transported by all ships under each individual CII setting remains constant at $2.02 \times 10^5$ tons, the distribution of loads among individual vessels varies significantly depending on the type of CII applied. Under Demand-based CII, only laden sailing legs contribute to the denominator, prompting more careful adjustments in cargo allocation and route selection to ensure compliance. In contrast, Supply-based CII is computed based on each ship's DWT and total sailing distance, often resulting in a different cargo distribution pattern. The fact that the total transported cargo remains unchanged across all CII schemes is due to the design of our instance, in which the total fleet capacity exceeds the total market demand, allowing all available cargo to be served regardless of the CII policy.

As shown in Figure \ref{fig:compare_carbon}, ships tend to emit more carbon under stricter Supply-based CII requirements. Total emissions are approximately $9 \times 10^{11} \, \text{g}$ under the lenient Demand-based and Supply-based CII, but increase to $1.0 \times 10^{12} \, \text{g}$ when the Supply-based CII becomes stricter, as illustrated in Figure \ref{fig:comparison_of_carbon_emissions_under_CII}. This increase results from additional ballast sailings undertaken to satisfy the tighter requirements, which in turn lead to higher emissions.

\begin{figure}[!ht]
    \centering
    \includegraphics[width=0.7\linewidth]{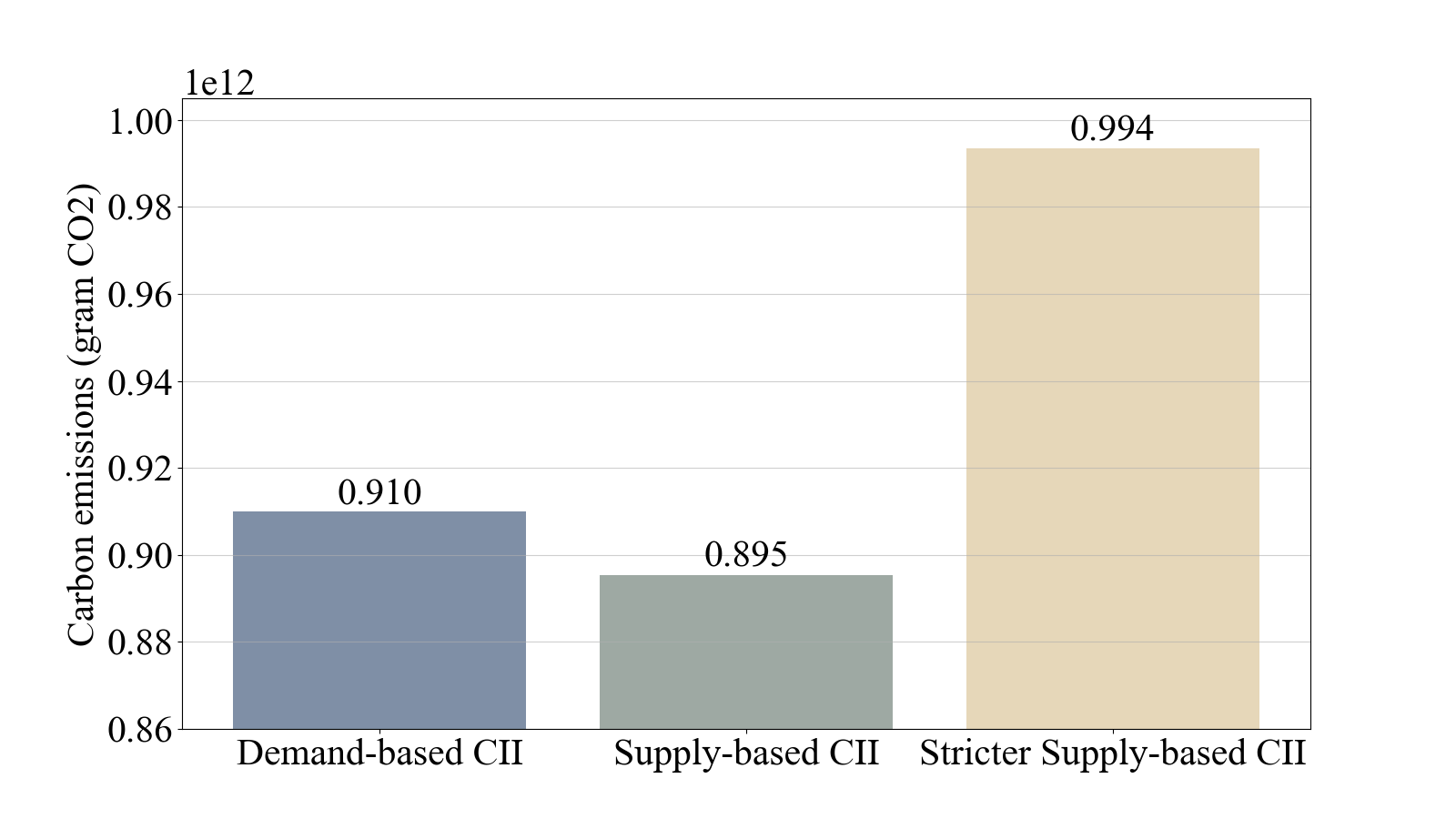}
    \caption{Total Carbon Emission under Different CII Forms}
    \label{fig:comparison_of_carbon_emissions_under_CII}
    \caption*{\raggedright \setstretch{1}
    Note: ``Demand-based CII" refers to the Demand-based Carbon Intensity Indicator calculated with a D rating, 
    ``Supply-based CII" refers to the Supply-based Carbon Intensity Indicator with a D rating, 
    and ``Stricter Supply-based CII" corresponds to the Supply-based Carbon Intensity Indicator with an A rating.}
\end{figure}

The opposite value of the objective function representing profits, shows that under the lenient CII standards, profits about $\$4.7 \times 10^6$ (Demand-based CII) and $\$6.6 \times 10^6$ (Supply-based CII), while under the stricter Supply-based CII, they drop to $\$-2.1 \times 10^6$, as shown in Figure \ref{fig:comparison_of_profits_under_CII}. This decline is partly due to excessive ballast sailing, which is more costly than remaining stationary, as cargo delivery and revenue remain consistent across scenarios.

\begin{figure}[!ht]
    \centering
    \includegraphics[width=0.7\linewidth]{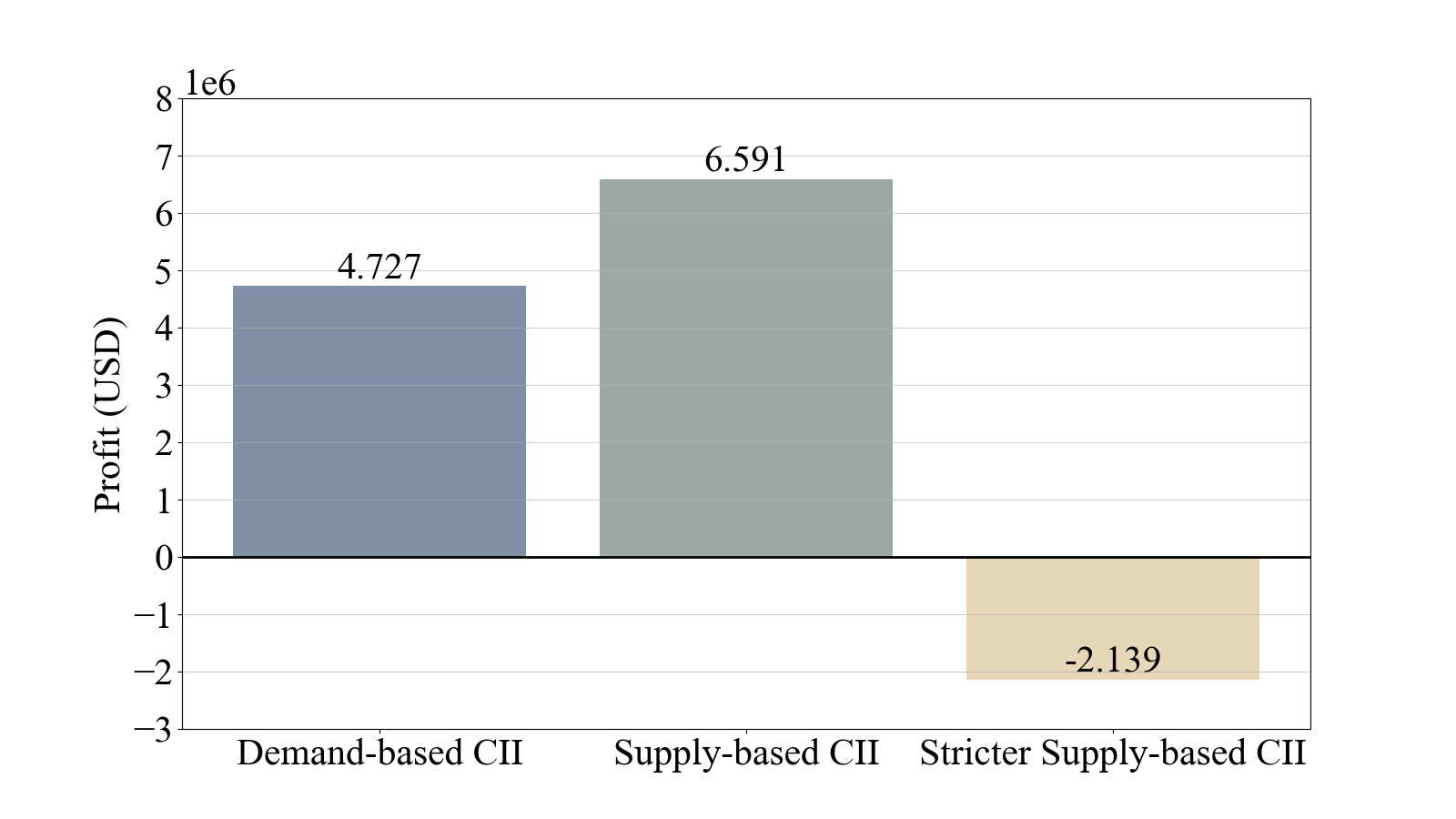}
    \caption{Total Profits under Different CII Forms}
    \caption*{\raggedright \setstretch{1}Note: ``Demand-based CII" refers to the Demand-based Carbon Intensity Indicator calculated with a D rating, 
    ``Supply-based CII" refers to the Supply-based Carbon Intensity Indicator with a D rating, 
    and ``Stricter Supply-based CII" corresponds to the Supply-based Carbon Intensity Indicator with an A rating.}
    \label{fig:comparison_of_profits_under_CII}
\end{figure}

Overall, we highlight the existence of the Supply-based CII paradox. To further illustrate how varying levels of CII stringency and ship emission intensities influence the severity of this paradox, we provide a detailed numerical experiment in Appendix~\ref{appendix:supply_cii_exp}. This appendix analyzes the evolution of total carbon emissions and profits under different Supply-based CII standards and emission levels, confirming the robustness of the paradox and offering additional managerial insights.

\subsection{Effect of stochastic scenarios}
\label{sec:ana_scenario}

In this section, Supply-based CII is selected as the environmental constraint to analyze the model. The following of this section compares the revenue gap using three formulations. 
\begin{itemize}
    \item Ignoring uncertainty: expected profits of using the solution with means (EMS). We assume that the random parameters in \textbf{P-2} all take their expected values.
    \item Considering uncertainty: expected profits of using the solution of the resource problem (RP). We solve the two-stage stochastic model taking into account all possible scenarios in \textbf{P-2}.
    \item Knowing uncertainty: expected profits of perfect information (Wait-and-See, WS). We calculate the maximum profit of all possible scenarios and then obtain their expected value.
\end{itemize}

\begin{figure}[!h]
    \centering
    \includegraphics[width=0.75\linewidth]{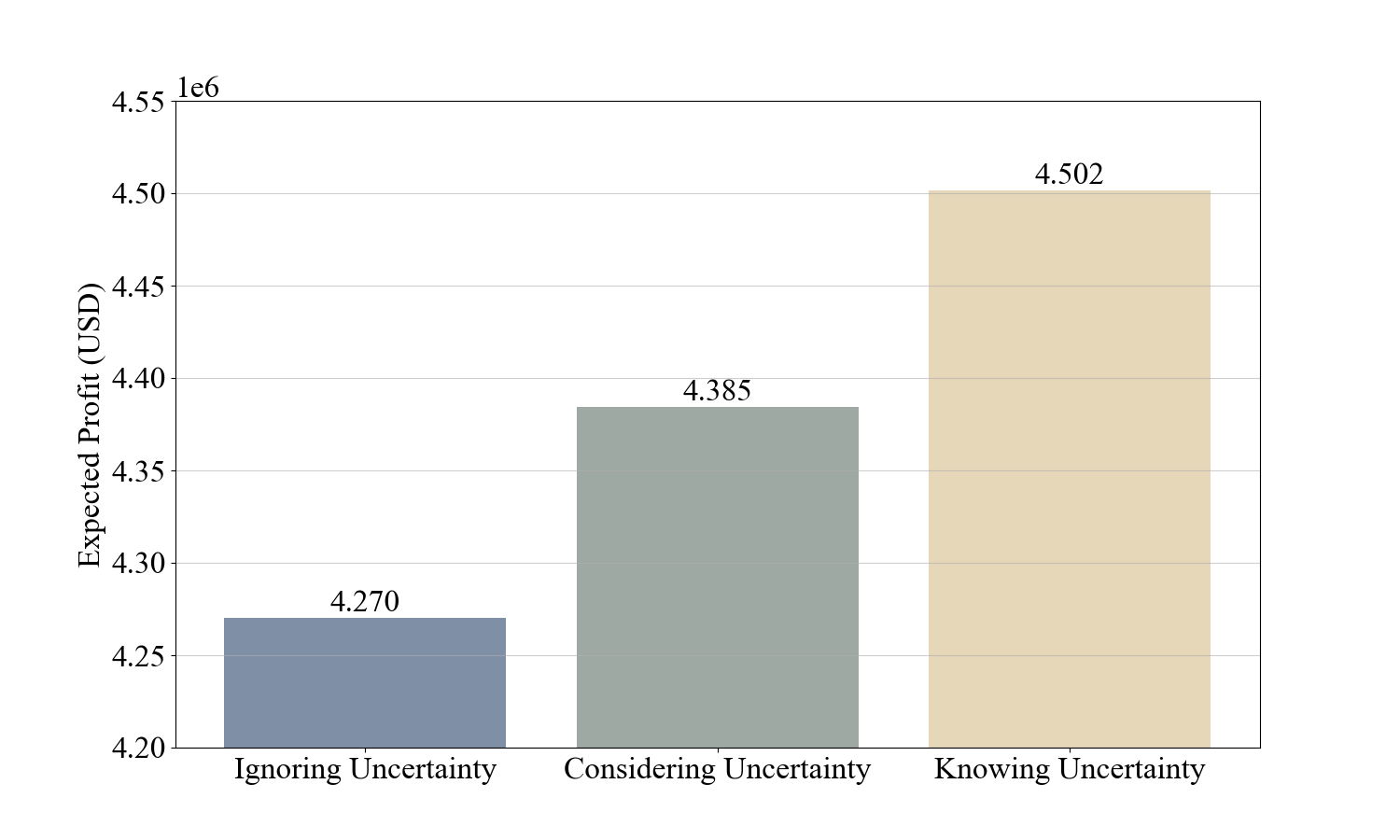}
    \caption{Comparison of Expected Profits}
    \label{fig:comparison_of_expected_profits}
\end{figure}

Figure \ref{fig:comparison_of_expected_profits} shows significant differences among the three formulations. The expected profits are $\$4.270 \times 10^6$ for \textit{Ignoring Uncertainty} (EMS), $\$4.385 \times 10^6$ for \textit{Considering Uncertainty} (RP), and $\$4.502 \times 10^6$ for \textit{Knowing Uncertainty} (WS). Two indicators are introduced to highlight the importance of uncertainty information:

\begin{itemize}
    \item \textbf{Expected Value of Perfect Information (EVPI)}: Measures the benefit of knowing the future with certainty. Defined as
    \begin{align}
        \text{EVPI} = \text{WS} - \text{RP}.
    \end{align}
    A positive EVPI indicates that knowing the future reduces the expected cost.

    \item \textbf{Value of the Stochastic Solution (VSS)}: Measures the benefit of incorporating the distribution of future outcomes. Defined as
    \begin{align}
        \text{VSS} = \text{RP} - \text{EMS}.
    \end{align}
    A positive VSS indicates that using the stochastic program improves upon the expected-value solution.
\end{itemize}

In this case, EVPI is $\$1.17 \times 10^5$, and VSS is $\$1.15 \times 10^5$. Both the EVPI and VSS are non-negative, consistent with Proposition~5.a in Chapter~4 of \cite{birge2011introduction}. We include the established theoretical proof of the non-negativity in Appendix~\ref{appendix:stochastic} for reference. A sensitivity analysis in Appendix~\ref{appendix:information_robust} further demonstrates that both EVPI and VSS are positive in our model. The positive VSS highlights the importance of considering future uncertainties in planning. The EVPI reflects the value of having perfect future knowledge relative to relying on probabilistic information. Thus, incorporating future uncertainties into planning reduces profit loss, even if it does not match the profit achievable with perfect information.

\subsection{Policy and managerial implications}
 
This section highlights policy and managerial implications associated with implementing the CII and incorporating stochastic market conditions into planning. 

\textbf{Implementation of CII}  
As discussed in Section \ref{sec:compare_cii}, inappropriate implementation of CII (with Supply-based CII examined under the settings of our model in this paper) can result in counterproductive environmental outcomes and reduced corporate profits, harming both the environment and market stability. Other forms of CII may also increase emissions in specific scenarios \citep{wang2021paradox}. To achieve Pareto improvements in market and environmental benefits, the IMO could refine the CII framework to better align with carbon reduction targets. The improved CII should discourage behaviors such as slowing down, detouring, or ballasting, which may reduce supply, disrupt the market, or even increase emissions. Instead, policies should guide ships toward genuine decarbonization through technological advancements such as alternative fuels and improved engine systems.

\textbf{Considering stochastic conditions}  
As shown in Section \ref{sec:ana_scenario}, incorporating future uncertainties into financial projections significantly enhances earnings. This pinpoints the importance of anticipating market fluctuations and integrating these dynamics into long-term tactical planning. Companies could leverage existing data and market intelligence to forecast changes and align their planning and decision-making accordingly. This approach enables enterprises to better adapt to volatility, seize opportunities, and optimize revenues. Conversely, ignoring market uncertainties can lead to substantial revenue losses and jeopardize a company's competitive position. Accounting for future uncertainties and developing agile strategies is vital for maintaining competitiveness and ensuring sustainable growth.

\section{Conclusion}
\label{sec:conclusion}
  
With the growing emphasis on decarbonization in the maritime industry, this study develop a two-stage stochastic programming model to optimize fleet deployment decisions for shipping companies under the CII policy in uncertain tramp shipping markets. Our study demonstrates the specific value of the tactical perspective: it reveals fleet-wide, year-long consequences of policy incentives (the ``CII paradox'') that would be invisible in a single-voyage analysis, and quantifies the economic benefit of incorporating long-term market uncertainty into planning. Unlike studies that focus primarily on market mechanisms such as the EU ETS or on technological development like the EEXI, our model evaluates the implications of the CII from the perspective of shipping companies. Our findings highlight critical issues in policy implementation and the value of incorporating market uncertainties into long-term decision-making to enhance profitability.

Numerical analysis under our model settings revealed potential paradoxes in which Supply-based CII could inadvertently increase carbon emissions and reduce corporate profits. This occurs when strategies such as ballasting or slowing down are used to meet lower CII standards, which reduce capacity, destabilize markets, and fail to achieve true emissions reduction. If the denominator of the CII metric does not accurately reflect actual operational behavior, failing to distinguish between laden, ballast, or idle states, distorted incentives may arise, harming both economic and environmental outcomes. Although our analysis focuses on tramp shipping, these insights are likely to generalize to other contexts where operational flexibility such as ballast sailing interacts with aggregated CII metrics, highlighting the need for more appropriate policies that balance environmental goals with economic realities.

Moreover, our study demonstrates the importance of incorporating future market trends into fleet deployment strategies. Proactive planning enables shipping companies to anticipate market shifts, optimize resource allocation, and maximize profits while mitigating risks from market volatility. While we used artificially assigned probability distributions and limited scenarios to describe future market conditions due to data limitations, the qualitative insight remains robust when generalized to large-scale real-world applications: shipping companies can always benefit from more accurate forecasts, although the exact magnitude of the benefit depends on each company's specific circumstances. This qualitative insight is inherently generalizable and does not diminish with problem scale, which is a constitutive advantage of the stochastic optimization approach.
 
This study has some limitations that suggest directions for future research. First, the model treats the CII as a strict constraint, whereas in real-world applications, moderate violations are often permitted with penalties. Future research could incorporate penalty costs for CII violations to better reflect practical implementation and guide the design of effective regulatory mechanisms. Second, while our model helps uncover paradoxes associated with the Supply-based CII, it could be extended to explore whether similar issues arise under alternative CII formulations, such as Demand-based CII.  Furthermore, the proposed method is limited to a heuristic route search algorithm. Future work could consider integrating this approach with other metaheuristics or domain-specific techniques to enhance solution quality and computational efficiency for solving the proposed complex model. Finally, while our study demonstrates the existence and positive value of incorporating future information within a realistic stochastic model, its precise magnitude in practice is contingent on the true distribution of market parameters. A promising direction for future research is to refine this quantification using actual industry data.

\clearpage
\bibliographystyle{apalike}
\bibliography{ref}

\newpage
\appendix

\section{Scenario design and validation}
\label{appendix:scenarios}

This appendix provides additional details on the scenario generation and validation process used in our study.

\subsection{Rationale behind determining the number of scenarios}

The scenario generation process involves certain subjective assumptions. However, these assumptions are guided by economic literature that explores plausible correlations between market demand, freight rates, and fuel prices. The 13 scenarios constructed in this study are designed to represent a comprehensive range of possible market conditions rather than to reproduce exact probabilistic distributions from historical data, which are often proprietary or unavailable.

A comparable approach can be found in \cite{wang2018planning}, where random parameters in the second stage are generated using multiplicative random factors. Specifically, for a given random parameter such as the freight rate ($RP$), the model in \cite{wang2018planning} defines
$$
RP = \xi_1 \cdot \mathbb{E}[RP],
$$
where $\mathbb{E}[RP]$ denotes the expected freight rate and $\xi_1$ represents an underlying random factor. Similarly, demand is generated as $D = \xi_2 \cdot \mathbb{E}[D]$, where $\xi_2$ is another random factor. To reflect the economic rationale that demand and price are correlated through market prosperity, a strong positive correlation is imposed between $\xi_1$ and $\xi_2$. In that study, 50 scenarios are generated to capture the randomness of 27 parameters (including 22 contract demands).

Our setting can be viewed as a simplification of the approach in \cite{wang2018planning}. In the second-stage model, parameters expected to remain stable are fixed at their base values in \textbf{P-2}, while those expected to increase or decrease are scaled to 120\% and 80\% of their base values, respectively. Following the structure in \cite{wang2018planning}, the random factor $\xi$ is defined with the discrete distribution: $\mathbb{P}(\xi = 1.2) = \tfrac{5}{13},\ \mathbb{P}(\xi = 1.0) = \tfrac{3}{13},\ \mathbb{P}(\xi = 0.8) = \tfrac{5}{13}$. These probabilities are not uniformly distributed because we intentionally include cases where fuel cost and market demand move in opposite directions, and where freight rate may increase, decrease, or remain unchanged. Each of the 13 scenarios is assigned equal probability ($1/13$) for simplicity, since historical data to reflect real scenario probabilities are not available. The correlations among fuel price, market demand, and freight rate are introduced deliberately according to insights from economic literature. Specifically, we model a positive correlation between fuel price and freight rate, reflecting the cost pass-through effect; a positive correlation between market demand and freight rate, driven by supply-demand dynamics; and a negative correlation between fuel price and market demand, accounting for potential cost-driven demand reduction. The resulting pairwise correlation coefficients among the three variables are summarized in Table~\ref{tab:A1}.

\begin{table}[!ht]
    \centering
    \caption{Correlation matrix among fuel price, market demand, and freight rate}
    \label{tab:A1}
    \begin{tabular}{lccc}
        \toprule
        & Fuel price & Market demand & Freight rate \\
        \midrule
        Fuel price    & 1.0 & -0.4 & 0.4 \\
        Market demand & -0.4 & 1.0 & 0.4 \\
        Freight rate  & 0.4 & 0.4 & 1.0 \\
        \bottomrule
    \end{tabular}
\end{table}

These correlations are consistent with market intuition: fuel price and freight rate are positively related due to cost pass-through; market demand and freight rate are positively related through demand–supply interaction; and fuel price and market demand exhibit a negative correlation because higher fuel costs may reduce demand through higher transport prices. Hence, the designed scenarios capture the main economic interactions observed in real markets.

\subsection{Validation with additional scenarios}

Given the limited number of contracts (two or three in our test cases), it is unnecessary to use a large number of scenarios (e.g., more than 50) as in \cite{wang2018planning}. To verify that 13 scenarios are sufficient, we performed additional experiments by generating 15 expanded datasets of 52 scenarios. The 52-scenario set consists of the original 13 scenarios plus 39 additional ones obtained by adding stochastic variations to each original scenario (three randomized variants per base scenario). The comparison results are shown in Table~\ref{tab:A2}.

\begin{table}[htbp]
\centering
\caption{Comparison of objective values, cost, and revenue under two scenario sets}
\label{tab:A2}
\resizebox{\textwidth}{!}{
\begin{tabular}{ccccccc}
\toprule
\multirow{2}{*}{\textbf{Instance}} 
& \multicolumn{2}{c}{\textbf{Objective Value (\$)}} 
& \multicolumn{2}{c}{\textbf{Cost (\$)}} 
& \multicolumn{2}{c}{\textbf{Revenue (\$)}} \\
\cmidrule(lr){2-3} \cmidrule(lr){4-5} \cmidrule(lr){6-7}
& \textbf{13 Scenarios} & \textbf{52 Scenarios} 
& \textbf{13 Scenarios} & \textbf{52 Scenarios} 
& \textbf{13 Scenarios} & \textbf{52 Scenarios} \\
\midrule
1  & $8.56{\times}10^6$  & $8.13{\times}10^6$  & $1.87{\times}10^7$ & $1.81{\times}10^7$ & $1.01{\times}10^7$ & $1.00{\times}10^7$ \\
2  & $9.59{\times}10^6$  & $9.16{\times}10^6$  & $2.08{\times}10^7$ & $2.02{\times}10^7$ & $1.12{\times}10^7$ & $1.10{\times}10^7$ \\
3  & $3.62{\times}10^6$  & $3.35{\times}10^6$  & $1.31{\times}10^7$ & $1.27{\times}10^7$ & $9.47{\times}10^6$ & $9.38{\times}10^6$ \\
4  & $8.55{\times}10^6$  & $8.11{\times}10^6$  & $1.45{\times}10^7$ & $1.41{\times}10^7$ & $5.98{\times}10^6$ & $6.00{\times}10^6$ \\
5  & $9.22{\times}10^6$  & $8.82{\times}10^6$  & $1.59{\times}10^7$ & $1.55{\times}10^7$ & $6.67{\times}10^6$ & $6.66{\times}10^6$ \\
6  & $-8.40{\times}10^5$ & $-1.24{\times}10^6$ & $1.35{\times}10^7$ & $1.31{\times}10^7$ & $1.43{\times}10^7$ & $1.44{\times}10^7$ \\
7  & $1.70{\times}10^6$  & $1.28{\times}10^6$  & $1.67{\times}10^7$ & $1.63{\times}10^7$ & $1.50{\times}10^7$ & $1.50{\times}10^7$ \\
8  & $-1.90{\times}10^6$ & $-2.38{\times}10^6$ & $1.57{\times}10^7$ & $1.53{\times}10^7$ & $1.76{\times}10^7$ & $1.76{\times}10^7$ \\
9  & $4.03{\times}10^5$  & $4.34{\times}10^4$  & $1.69{\times}10^7$ & $1.65{\times}10^7$ & $1.65{\times}10^7$ & $1.64{\times}10^7$ \\
10 & $6.98{\times}10^6$  & $6.60{\times}10^6$  & $1.50{\times}10^7$ & $1.46{\times}10^7$ & $8.03{\times}10^6$ & $7.97{\times}10^6$ \\
11 & $8.26{\times}10^4$  & $3.31{\times}10^4$  & $1.69{\times}10^7$ & $1.69{\times}10^7$ & $1.68{\times}10^7$ & $1.69{\times}10^7$ \\
12 & $4.11{\times}10^6$  & $3.65{\times}10^6$  & $1.87{\times}10^7$ & $1.82{\times}10^7$ & $1.46{\times}10^7$ & $1.46{\times}10^7$ \\
13 & $-5.22{\times}10^6$ & $-5.90{\times}10^6$ & $1.67{\times}10^7$ & $1.61{\times}10^7$ & $2.20{\times}10^7$ & $2.20{\times}10^7$ \\
14 & $-6.27{\times}10^6$ & $-7.14{\times}10^6$ & $1.90{\times}10^7$ & $1.81{\times}10^7$ & $2.53{\times}10^7$ & $2.52{\times}10^7$ \\
\bottomrule
\end{tabular}}
\end{table}

The results show that the differences between the two models are marginal. On average, the model with 13 scenarios yields an objective value of $2.76\times10^6$, while that with 52 scenarios yields $2.32\times10^6$. The average cost and revenue are $1.66\times10^7$ and $1.38\times10^7$, respectively, for the 13-scenario model, and $1.61\times10^7$ and $1.38\times10^7$ for the 52-scenario model. The additional 39 scenarios thus have limited influence on the optimal results while substantially increasing computational complexity. Hence, we retain the 13-scenario setting in the main analysis.

\section{Detailed analysis of the Supply-based CII paradox}
\label{appendix:supply_cii}

\subsection{Formal proof of the Supply-based CII paradox}
\label{appendix:supply_cii_theory}

This appendix provides a formal proof that, under mild conditions, tightening the Supply-based CII standard can paradoxically lead to an increase in total annual emissions for a single vessel. We work with the Supply-based CII definition: $$\text{CII}=\frac{M}{C\cdot D^T},$$ 
where $\text{CII}$ denotes the Supply-based CII, \(M\) is annual CO2 emissions (grams), \(C\) is the capacity measure (tonnes), and \(D^T\) is the total annual distance (nautical miles). We begin by stating the necessary assumptions and the main proposition.

\paragraph{Assumptions}
\begin{enumerate}[label=(A\arabic*)]
    \item \textbf{Emission Rates:} The carbon emission rate per nautical mile is higher when the ship is laden than when it is in ballast: $\epsilon_{\text{laden}} > \epsilon_{\text{ballast}} > 0$. The emission rate per unit time at port is $\epsilon_{\text{port}} \ge 0$.
    \item \textbf{Port vs. Ballast Emission Intensity:} The emission intensity of ballast sailing is greater than that of port stay when measured per unit of CII-accounted activity: $\epsilon_{\text{ballast}} > \epsilon_{\text{port}} / e$, where $e$ is the sailing speed.
    \item \textbf{Fixed Laden Task:} The total laden distance $D^T_{\text{laden}}$ and its associated emissions $M_{\text{laden}}$ are fixed. The ship is initially compliant with the baseline CII standard, meaning that it operates at the lowest sailing speed during both laden and ballast voyages.
    \item \textbf{Feasible Standard:} The CII standard is feasible, i.e., it is not stricter than the lowest achievable intensity of continuous ballast sailing: $C\cdot\mathrm{CII}_{\text{target}} \ge \epsilon_{\text{ballast}}$, where $C$ is the capacity of the ship (tonnage).
\end{enumerate}

\begin{proposition}
Under Assumptions (A1)-(A4), if a ship is initially compliant with a Supply-based CII standard $\mathrm{CII}_0$ and the standard is tightened to $(1-p)\mathrm{CII}_0$ for some $p \in (0,1)$, then the ship can achieve compliance with the new standard by substituting port time for additional ballast sailing, but this strategy will result in an increase in its total annual CO2 emissions.
\end{proposition}

\begin{proof}
Let the baseline (pre-tightening) annual emissions and distance be:
\[
M_0 = M_{\text{laden}} + M_{\text{ballast},0} + M_{\text{port},0}, \quad D_0 = D^T_{\text{laden}} + D^T_{\text{ballast},0} + D^T_{\text{port},0},
\]
with $\mathrm{CII}_0 = M_0 / (C \cdot D_0)$. 

Suppose the ship increases its annual ballast distance by $\Delta > 0$. This requires $\Delta / e$ additional hours of ballast sailing, which we assume is offset by an equal reduction in port time. By Assumption (A1) and (A3), the new emissions and distance are:
\[
M_{\text{new}} = M_0 + \epsilon_{\text{ballast}} \Delta - (\epsilon_{\text{port}} / e) \Delta, \quad D^T_{\text{new}} = D_0 + \Delta.
\]

The ship complies with the new standard if:
\[
\frac{M_{\text{new}}}{C \cdot D^T_{\text{new}}} = (1-p)\mathrm{CII}_0=(1-p)\frac{M_0}{C\cdot D_0^T}.
\]
Substituting the expressions and solving for $\Delta$ yields:
\[
\Delta = \frac{p M_0}{(1-p)\frac{M_0}{D_0^T} - \left(\epsilon_{\text{ballast}} - \frac{\epsilon_{\text{port}}}{e}\right)}.
\]

By Assumption (A4), $(1-p)\frac{M_0}{D^T_0} \ge \epsilon_{\text{ballast}} > \epsilon_{\text{ballast}} - \frac{\epsilon_{\text{port}}}{e}$. Therefore, the denominator is positive, guaranteeing $\Delta > 0$.

The change in total emissions is:
\[
M_{\text{new}} - M_0 = \left(\epsilon_{\text{ballast}} - \frac{\epsilon_{\text{port}}}{e}\right) \Delta.
\]
Assumption (A2) ensures that this quantity is positive. Therefore, total emissions increase.
\end{proof}

This proof formalizes the perverse incentive inherent in the Supply-based CII. By rewarding total distance rather than productive work, the regulation makes ballast sailing a viable compliance strategy. Since ballast sailing has a higher operational carbon intensity than staying at port ($\epsilon_{\text{ballast}} > \epsilon_{\text{port}}/e$), this substitution increases overall emissions.

\subsection{Evolution of Carbon Emissions and Profits under Different Supply-based CII Standards}
\label{appendix:supply_cii_exp}

To validate the theoretical insights presented above, we perform a sensitivity analysis on different Supply-based CII standards and ship carbon emission levels to examine the robustness of the Supply-based CII paradox.

The CII standard is varied from Rate D (8.6 g/(ton$\cdot$nmile)) to Rate A (11.8 g/(ton$\cdot$nmile)). Ship carbon emission levels are adjusted to 100\%, 90\%, 80\%, and 70\% to simulate the adoption of low-carbon technologies, reflecting reduced emissions per nautical mile at the same speed and cargo load.

Our results indicate two key patterns. First, the Supply-based CII paradox—where stricter standards can lead to higher total emissions—becomes more pronounced as the CII standard tightens. Second, the paradox is mitigated when ships adopt low-carbon technologies. As illustrated in Figure~\ref{fig:CII_robust}, both a more lenient CII standard and lower emission levels reduce the paradox. Specifically, for a given emission level, the increase in total emissions caused by stricter standards diminishes. Likewise, for a given standard, lower ship emissions lead to smaller increases in total emissions.

A similar trend is observed for profits: stricter CII standards reduce profitability, whereas lower ship emission levels improve it. This occurs because high-emission ships are incentivized to increase ballast sailing to comply with stricter CII standards, raising total emissions and lowering profits. In contrast, low-emission ships or ships operating under more lenient standards have less need for ballast sailing, resulting in lower total emissions and higher profits.

\begin{figure}[!h]
    \begin{subfigure}[b]{0.5\textwidth}
        \centering
        \includegraphics[width=\linewidth]{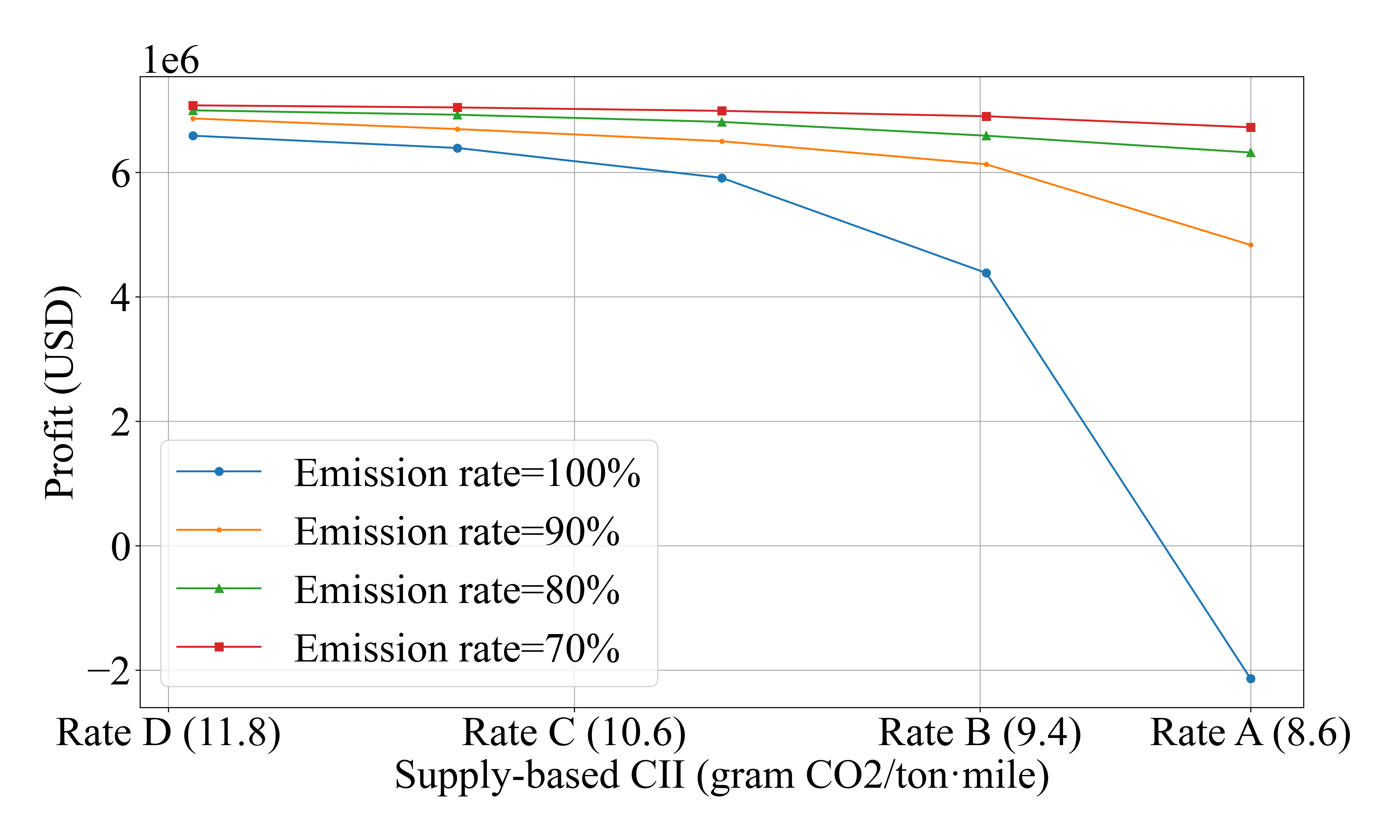}
        \caption{Profits}
        \label{fig:CII_robust_profit}
    \end{subfigure}%
    \hfill
    \begin{subfigure}[b]{0.5\textwidth}
        \centering
        \includegraphics[width=\linewidth]{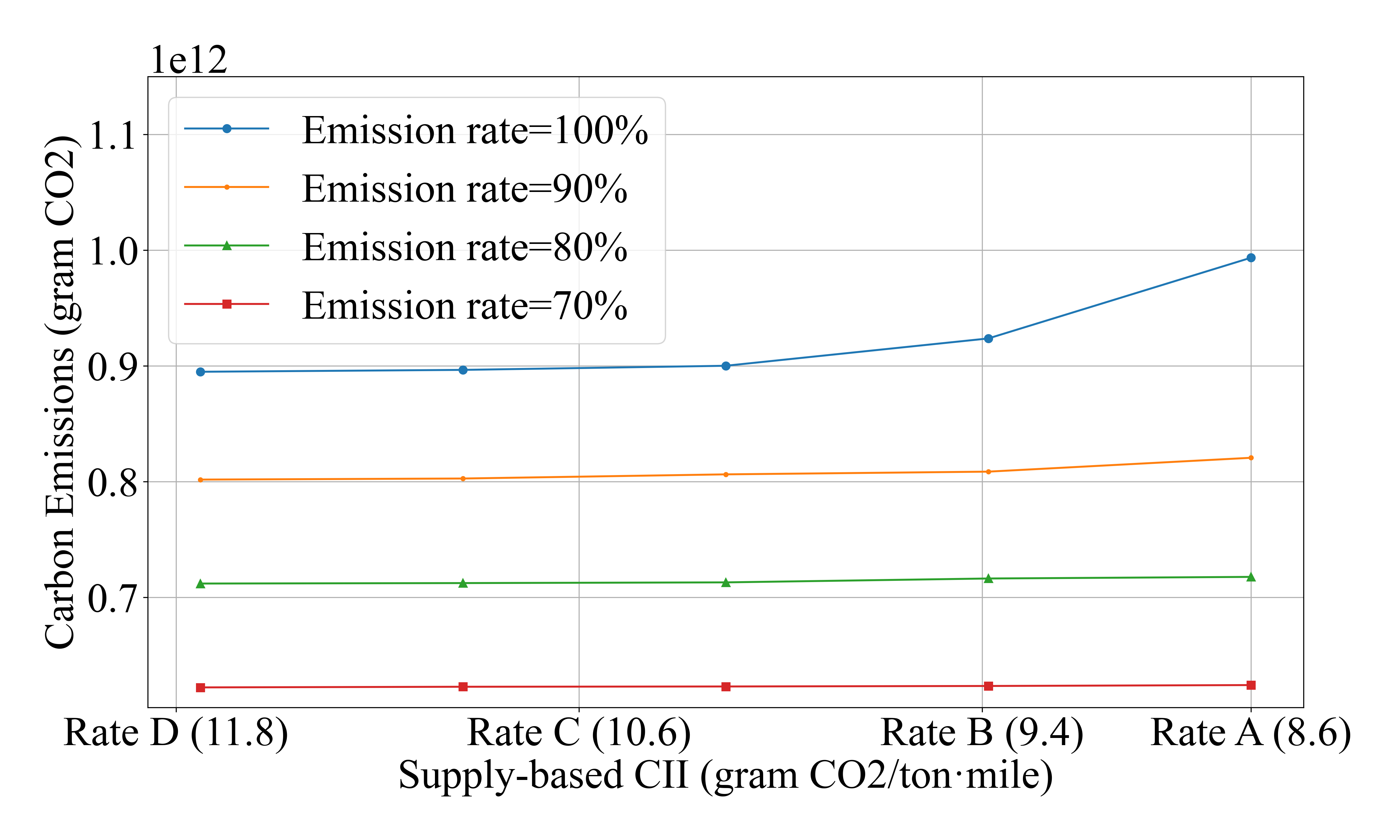}
        \caption{Carbon emissions}
        \label{fig:CII_robust_carbon}
    \end{subfigure}
    \caption{Changes in (a) profits and (b) carbon emissions under different CII standards and emission levels.}
    \label{fig:CII_robust}
\end{figure}

\section{The value of information and the stochastic solution}

\subsection{Proof of the Non-Negativity of the Value of Future Information}
\label{appendix:stochastic}

As a basic property of stochastic programming, the value of information, namely the Expected Value of Perfect Information (EVPI) and the Value of the Stochastic Solution (VSS), is non-negative, as established by \cite{birge2011introduction}. Therefore, for the reader's convenience, we reproduce the detailed proof here to provide a theoretical foundation for the numerical analysis in Section~\ref{sec:ana_scenario}.

\begin{proposition}[Non-negativity of EVPI and VSS]
Consider a two-stage stochastic programming problem with a maximization objective, and with first-stage decision $x \in X$ and random scenario $\xi$. Let $v(x,\xi)$ denote the optimal payoff given $x$ and $\xi$. Define
\begin{align*}
    z^{\text{RP}} &:= \max_{x \in X} \mathbb{E}_\xi[v(x,\xi)], \\
    z^{\text{WS}} &:= \mathbb{E}_\xi\left[ \max_{x \in X} v(x,\xi) \right], \\
    z^{\text{EMS}} &:= \mathbb{E}_\xi[v(x^{\text{EV}},\xi)], \quad x^{\text{EV}} \in \arg\max_{x \in X} v(x,\mathbb{E}[\xi]).
\end{align*}
Then the Expected Value of Perfect Information (EVPI) and the Value of the Stochastic Solution (VSS) satisfy
$$\text{EVPI} := z^{\text{WS}} - z^{\text{RP}} \ge 0, \qquad
\text{VSS} := z^{\text{RP}} - z^{\text{EMS}} \ge 0.$$

\end{proposition}

\begin{proof}
\noindent \textbf{(1) EVPI is non-negative}\\
For any scenario $\xi$, the following inequality holds by definition of the maximum:
$$\max_{x \in X} v(x,\xi) \ge v(x^\star,\xi), \quad \forall x^\star \in X.$$
Let $x^\star = x^{\text{RP}}$, the optimal solution of the Resource Problem (RP). Taking expectation over $\xi$ gives
$$\mathbb{E}_\xi \left[ \max_{x \in X} v(x,\xi) \right] \ge \mathbb{E}_\xi [v(x^{\text{RP}},\xi)] = z^{\text{RP}}.$$
By definition, the left-hand side is $z^{\text{WS}}$. Therefore,
$$z^{\text{WS}} - z^{\text{RP}} \ge 0 \quad \Longrightarrow \quad \text{EVPI} \ge 0.$$

\noindent \textbf{(2) VSS is non-negative}\\
By definition of $z^{\text{RP}}$ as the maximum expected payoff,
$$z^{\text{RP}} = \max_{x \in X} \mathbb{E}_\xi[v(x,\xi)] \ge \mathbb{E}_\xi[v(x^{\text{EV}},\xi)] = z^{\text{EMS}}.$$
Thus,
$$z^{\text{RP}} - z^{\text{EMS}} \ge 0 \quad \Longrightarrow \quad \text{VSS} \ge 0.$$
\end{proof}

So far, we have shown the non-negativity of EVPI and VSS. We now analyze the equality conditions.
\begin{itemize}
    \item $\text{EVPI} = 0$ if and only if the RP solution $x^{\text{RP}}$ is optimal for every scenario $\xi$, i.e.,
    $$\max_{x \in X} v(x,\xi) = v(x^{\text{RP}},\xi) \quad \text{almost surely}.$$
    This corresponds to the case where knowing the exact scenario provides no additional benefit.
    \item $\text{VSS} = 0$ if and only if the solution $x^{\text{EV}}$ obtained from the expected scenario is already optimal in expectation over all scenarios, i.e.,
    $$x^{\text{EV}} \in \arg\max_{x \in X} \mathbb{E}_\xi[v(x,\xi)].$$
    In practice, $\text{VSS} > 0$ whenever the stochasticity of $\xi$ affects the optimal first-stage decision.
\end{itemize}

This result aligns with Proposition~5.a in Chapter 4 of \cite{birge2011introduction}, which considers a minimization objective function. Consequently, our numerical findings, which quantify the value of information in the tramp shipping market based on the given dataset, are theoretically justified and robust.

\subsection{Non-negativity of EVPI and VSS under different scenario settings}
\label{appendix:information_robust}

To validate the theoretical results presented above and the qualitative findings in the main text, we conduct a sensitivity analysis under various scenario settings. Specifically, based on the 13 baseline scenarios used in the numerical experiments, we generate additional sets of 26, 39, and 52 scenarios by introducing random perturbations to the original scenario parameters, such as demand, fuel price, and revenue. The perturbations follow a uniform distribution within $\pm15\%$ of the baseline values.

The results, summarized in Figure~\ref{fig:information_robust}, show that under different numbers and distributions of scenarios, both EVPI and VSS remain strictly positive, although their magnitudes vary slightly. This consistency verifies that the theoretical property of non-negativity is robust to scenario design. From a managerial perspective, this implies that access to better information or stochastic optimization always improves, or at least never worsens, the expected profit of shipping companies.

\begin{figure}[!h]
    \centering
    \includegraphics[width=0.75\linewidth]{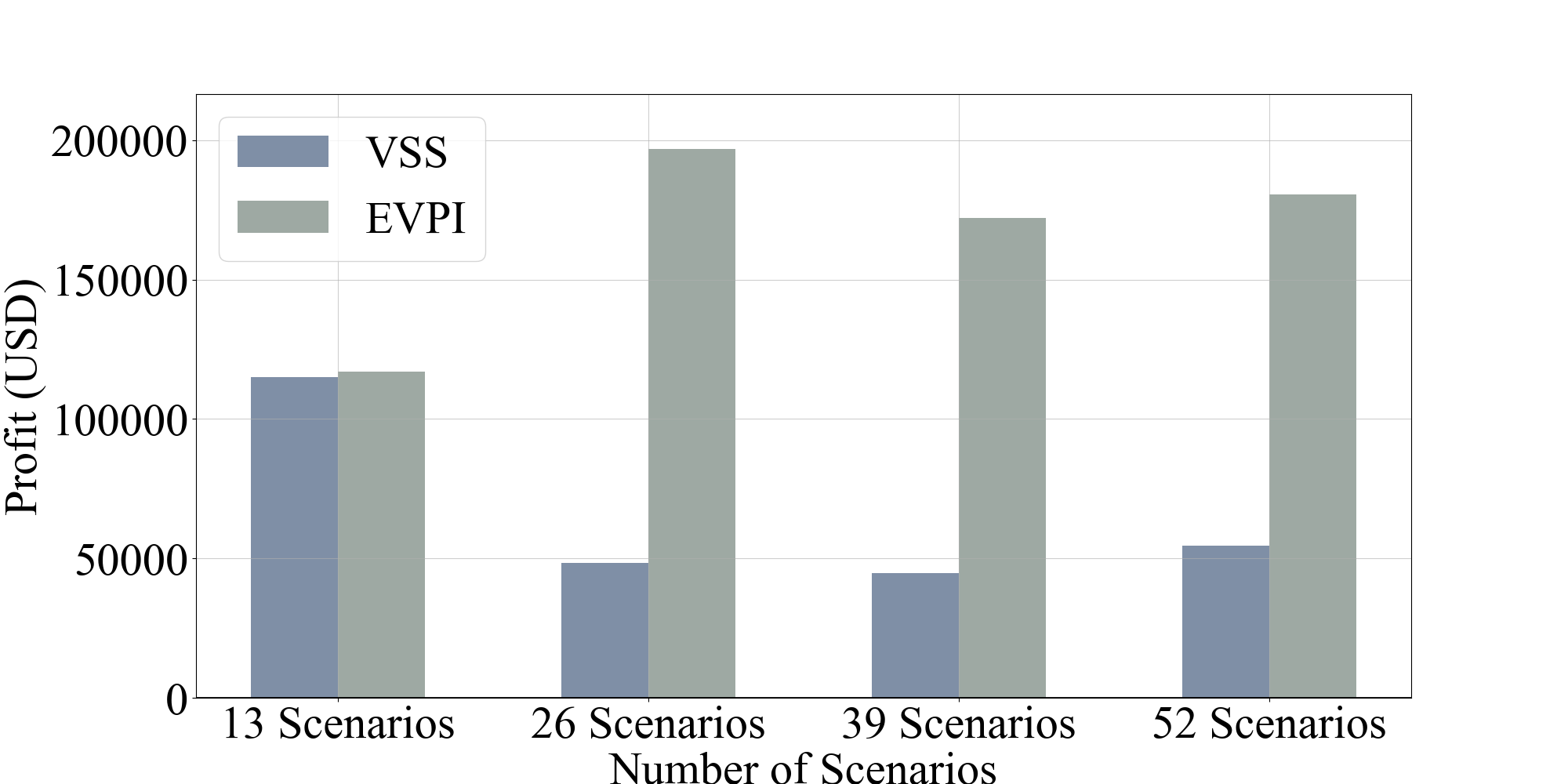}
    \caption{The value of future information: EVPI and VSS}
    \label{fig:information_robust}
    \caption*{\raggedright \setstretch{1}
    Note: ``EVPI'' refers to the Expected Value of Perfect Information; 
    ``VSS'' refers to the Value of the Stochastic Solution.}
\end{figure}

\end{document}